\documentclass[
12pt, % Main document font size
a4paper, % Paper type, use 'letterpaper' for US Letter paper
oneside, % One page layout (no page indentation)
headinclude,footinclude, % Extra spacing for the header and footer
BCOR5mm, % Binding correction
]{article}

\usepackage[
nochapters, % Turn off chapters since this is an article        
beramono, % Use the Bera Mono font for monospaced text (\texttt)
eulermath,% Use the Euler font for mathematics
pdfspacing, % Makes use of pdftex’ letter spacing capabilities via the microtype package
dottedtoc % Dotted lines leading to the page numbers in the table of contents
]{classicthesis} % The layout is based on the Classic Thesis style

\usepackage{graphicx} % Required for including images
\usepackage{enumitem} % Required for manipulating the whitespace between and within lists

\usepackage{amsmath,amssymb,amsthm} % For including math equations, theorems, symbols, etc

\usepackage{varioref} % More descriptive referencing

\hypersetup{
colorlinks=true, breaklinks=true, bookmarks=true,bookmarksnumbered,
urlcolor=webbrown, linkcolor=RoyalBlue, citecolor=webgreen, % Link colors
pdftitle={}, % PDF title
pdfauthor={\textcopyright}, % PDF Author
pdfsubject={}, % PDF Subject
pdfkeywords={}, % PDF Keywords
pdfcreator={pdfLaTeX}, % PDF Creator
pdfproducer={LaTeX with hyperref and ClassicThesis} % PDF producer
} % Include the structure.tex file which specified the document structure and layout

\usepackage{bm, stmaryrd}											
\usepackage{a4wide} 
\usepackage[english]{babel}

\newtheorem{thm}{Theorem}[section]
\newtheorem{lem}[thm]{Lemma}
\newtheorem{prop}[thm]{Proposition}

\newtheorem{coro}[thm]{Corollary}

\theoremstyle{definition}				%Remarques et notations en non italique

\numberwithin{equation}{section}

\newcommand{\A}{\mathbb{A}}
\newcommand{\B}{\mathbb{B}}
\newcommand{\R}{\mathbb{R}}

\newcommand{\N}{\mathbb{N}}
\newcommand{\Z}{\mathbb{Z}}
\newcommand{\Q}{\mathbb{Q}}

\newcommand{\X}{\mathbb{X}}
\newcommand{\Pj}{\mathbb{P}}
\newcommand{\Sph}{\mathbb{S}}
\newcommand{\V}{\detokenize{Vol}}

\newcommand{\dist}{\detokenize{dist}}
\newcommand{\h}{\detokenize{height}}
\newcommand{\bh}{\textbf{\detokenize{height}}}
\newcommand{\bdeg}{\textbf{\detokenize{deg}}}
\newcommand{\im}{\detokenize{Im}}
\title{\normalfont\spacedallcaps{Probabilistic Effectivity in the Subspace Theorem}} % The article title

\author{\spacedlowsmallcaps{Faustin ADICEAM\textsuperscript{1} \& \spacedlowsmallcaps{Victor SHIRANDAMI\textsuperscript{2}}} }% The article author(s) - author affiliations need to be specified in the AUTHOR AFFILIATIONS block

\date{} % An optional date to appear under the author(s)

\begin{document}

\maketitle

%----------------------------------------------------------------------------------------
%	HEADERS
%----------------------------------------------------------------------------------------

\renewcommand{\sectionmark}[1]{\markright{\spacedlowsmallcaps{#1}}} % The header for all pages (oneside) or for even pages (twoside)
\lehead{\mbox{\llap{\small\thepage\kern1em\color{halfgray} \vline}\color{halfgray}\hspace{0.5em}\rightmark\hfil}} % The header style

\pagestyle{scrheadings} % Enable the headers specified in this block

\begin{abstract}
\noindent The Subspace Theorem due to Schmidt (1972) is a broad generalisation of Roth's Theorem in Diophantine Approximation (1955) which, in the same way as the latter, suffers a notorious lack of effectivity. This problem is tackled from a probabilistic standpoint by determining the proportion of algebraic linear forms of bounded heights and degrees for which there exists a solution to the Subspace Inequality lying in a subspace of large height.  The estimates are established for a class of height functions emerging from an analytic parametrisation of the projective space.  They are pertinent in the regime where the heights of the algebraic quantities are larger than those of the rational solutions to the inequality under consideration, and are valid for approximation functions more general than the power functions intervening in the original Subspace Theorem. These estimates are further refined in the case of Roth's Theorem so as to yield a Khintchin--type density version of the so--called Waldschmidt conjecture (which is  known to fail pointwise). This answers a question raised by Beresnevich, Bernik and Dodson (2009).
\end{abstract}

\tableofcontents

\let\thefootnote\relax\footnotetext{\textsuperscript{1} {Laboratoire d’analyse et de mathématiques appliquées (LAMA), Université Paris-Est Créteil, France,} \texttt{faustin.adiceam@u-pec.fr}}
\let\thefootnote\relax\footnotetext{\textsuperscript{2}Department of Mathematics, The University of Manchester, United-Kingdom,  \texttt{ victor.shirandami@manchester.ac.uk}. The support of the Heilbronn Institute of Mathematical Research through the UKRI grant: \emph{Additional Funding Programme for Mathematical Sciences} (EP/V521917/1) is gratefully acknowledged.}

\section{Introduction}

\subsection{Effectivity in the Subspace Theorem.}\label{mainpbsec}

Given an integer $n\ge 2$ and vectors $\bm{x}, \bm{y}\in\R^n$, denote by $\bm{x\cdot y}$ the scalar product between $\bm{x}$ and $ \bm{y}$. In its initial version, the celebrated   Subspace Theorem due to Schmidt~\cite{schmidt} states the following~:

\begin{thm}[The Subspace Theorem --- Schmidt, 1972]\label{subbasis}
Let $ \bm{\alpha}_1, \dots, \bm{\alpha}_n\in\R^n$ be $n$ linearly independant vectors with algebraic coordinates. Fix $\varepsilon>0$. Then, the set of solutions in $\bm{q}\in\Z^n\backslash \{\bm{0}\}$ to the inequality 
\begin{equation}\label{subsp}
\prod_{i=1}^n \left|\bm{q}\bm{\cdot}\bm{\alpha}_i\right|\;\le\; \left\|\bm{q}\right\|^{-\varepsilon} 
\end{equation}
lies in a finite union of proper rational subspaces of $\Q^n$.
\end{thm}

\noindent Generalisations including non--Archimedean absolute values are known ---  see~\cite{bombieri} for further details. An upper bound for the \emph{number} of rational subspaces in this statement can be effectively determined --- see, e.g., \cite{schmidtbis} and the related references. One of the most fundamental open problems in Diophantine Analysis is to establish an effective version of this statement; namely, to determine the maximal height attained by the subspaces in the finite union.\\ 

\noindent The overarching goal undertaken  here is to  tackle this effectivity problem  from a probabilistic point of view by estimating the proportion of algebraic vectors for which the Subspace Inequality admits a large solution; that is, a solution lying in a subspace with large height. Given an integers $Q\ge 1$, this probabilistic approach can be very broadly described as such~: 
\begin{quote} (\textbf{Main Problem --- General Formulation}) \emph{What is the proportion of the set of $n$--tuples of algebraic vectors in $\R^n$ of fixed degrees such that the Subspace Inequality~\eqref{subsp} admits a nonzero integer solution $\bm{q}$ lying in a rational subspace of height at least $Q$?}\end{quote}

\noindent  A number of terms needs to be specified in this question --- this will be the purpose of the rest of the introduction culminating with the statement of the main Theorems~\ref{introeffeoth} and~\ref{introeffesubs}. For the time being, the following points should be noted~:

\begin{itemize}
\item the notion of "proportion of algebraic vectors of fixed degrees" is interpreted in terms of a density function as defined in \S\S\ref{secrothintro} \& \ref{genprobestimsubs}. It relies on the  definition of a suitable class of height and degree functions for the set of algebraic vectors, which is introduced in \S\ref{heightsdegalgvec};% and essentially obtained from a stereographic parametrisation of the projective space;

\item Estimates in terms  of density are derived  asymptotically from counting problems dealing  with finite point sets. To formalise this approach, together with the integer $Q=Q_1$ measuring the height of the rational subspace appearing in the statement of the Main Problem, another two parameters are introduced~: one, denoted by $H\ge 1$, measures the maximal height of the algebraic vectors  $\bm{\alpha}_i$ of fixed degrees, and another one, denoted by $Q_2$, bounds the height of the rational solutions  $\bm{q}$ lying in the subspace.

%\item the interplay between the \emph{upper bound} on the height  $Q_2$ of the rational solution $\bm{q}$ and the \emph{lower bound}  $Q_1$ of the rational subspace it belongs to is naturally in the present context. In some sense, it enables one to count "genuinely large" solutions to  inequality~\eqref{subsp} in   that they lie in structured sets, namely subspaces, enjoying a high complexity quantified through the value of $Q_1$. For instance, in the particular case where the Subspace Theorem  is specialised to Roth's Theorem in Diophantine Approximation (see \S~\ref{} for further details), this reduces to imposing independently \emph{both} the \emph{lower} bound $Q_1$ \emph{and} the \emph{upper} bound $Q_2$ on the denominator $q\ge 1$ of some irreducible rational solution to the Diophantine inequality under consideration;

\item the Main Problem is solved in a quantitative form so as to be able to answer a question such as the following~: \begin{quote}{\emph{How big should  the real quantities $H, Q_1$ and $Q_2$    be chosen under the assumption that  $H\ge Q_2\ge Q_1$ to guarantee that at most, say 5\%, of the $n$--tuples of algebraic vectors $\bm{\alpha}_i\in\R^n$ of given degrees and of heights at most $H$   yield the existence of a  solution $\bm{q}$ to the Subspace Inequality~\eqref{subsp} meeting these two pro\-per\-ties~: \textbf{(i)} $\bm{q}$ has height at most $Q_2$ and \textbf{(ii)} $\bm{q}$  lies in some  rational subspace of height at least $Q_1$?}} \end{quote}
Of paramount importance is the interplay between the parameters $H$ (measuring the heights of the algebraic vectors $\bm{\alpha}_i$) and $Q_2$ (measuring the heights of the rational solutions $\bm{q}$). The main results  provide a solution to this question when the parameter $H$ is  larger than $Q_2$~: this situation, involving the study of the statistical distribution of algebraic vectors, is what is referred to here as the \emph{probabilistic regime}. The other relevant alternative, viz.~when the pa\-ra\-meter $Q_2$ is  larger than $H$, can be referred to as the \emph{pseudo-deterministic regime} and will not be broached here\textsuperscript{3}.  

\let\thefootnote\relax\footnotetext{\textsuperscript{3} {The pseudo-deterministic regime essentially comes down to taking into account all possible rational solutions $\bm{q}$ to the Subspace Inequality lying in some subspace of height at least $Q_1$. Solving the Main Problem in this regime is, in fact,   little different from the original problem of effectivity~:  it reduces to finding an effective bound for the heights of the rational solutions  in the Subspace Theorem valid for all but  a small and quantified set of exceptional $n$-tuples of algebraic vectors of bounded height.}}
\end{itemize}

\noindent The particular case where the Subspace Theorem  is specialised to Roth's Theorem in Diophantine Approximation is considered with a special interest in the following section. The above approach is then refined, and the obtained results are more complete and more apparent.% in this case. They indeed yield precise estimates on the cardinality of the set of algebraic numbers of fixed degree and bounded  height admitting as a solution  to the Diophantine inequality under consideration an irreducible rational with denominator $q$   in the interval $\left[Q_1, Q_2 \right]$. 

  \subsection{Probabilistic Estimates in Roth's Theorem.}\label{secrothintro}
 
\noindent Roth's Theorem  amounts to considering the case  $n=2$ and to setting $\bm{\alpha}_1=(1,0)$ and $\bm{\alpha}_2=(\alpha,1)$ in Theorem~\ref{subbasis}~:
 
 \begin{thm}[Roth, 1955]\label{roth}
 Assume that $\alpha$ is an irrational algebraic number. Then, given any $\varepsilon>0$, the number of solutions in integers $q\ge 1$ and $p\in\Z$ to the inequality
 \begin{equation}\label{rothineq}
\left|\alpha-\frac{p}{q}\right|\;\le \; \frac{\Psi(q)}{q}
\end{equation}
is finite when $\Psi(q)=q^{-\varepsilon}$.
 \end{thm}  
 
\noindent More generally, let $\Psi~: \R_+\mapsto \R_+$ be an approximation function; that is, a nonincreasing function tending to zero at infinity. The Main Problem is considered in the generality where Roth's inequality~\eqref{rothineq} is satisfied for any approximation function satisfying a convergence or divergence condition (of which the original case $\Psi(x)=x^{-\varepsilon}$ is a particular example). To this end, given a fixed degree $d\ge 2$ and an integer $H\ge 1$, let $\A_d(H)$ be the set of \emph{real} algebraic numbers of degree $d$ and of (naive) height (defined as the maximum of the absolute values of the coefficients in its minimal polynomial) bounded by $H$. \\

\noindent Given integers $Q_2>Q_2\ge 1$, in view of inequality~\eqref{rothineq}, the goal is to enumerate the elements of the set $\A_d(H)$ which, when reduced modulo one, fall  inside the union of intervals
 \begin{equation}\label{jpsiroth}
J_{\Psi}(Q_1, Q_2)\;=\; \bigcup_{Q_1\le q< Q_2}\bigcup_{0\le p\le q} B^*\left(\frac{p}{q}, \frac{\Psi(q)}{q}\right).
\end{equation}
Here and throughout, for $x\in\R$ and $r>0$, one sets 
\begin{equation}\label{defballunit}
B^*(x, r)=\left[x-r, x+r\right]\cap [0,1].
\end{equation} 
By computing explicitly the implicit constants involved in the classical sieving argument outlined in~\cite[p.25--28]{sprin}, it is not hard to see that, under the assumption that the approximation function $\Psi$ is nonincreasing,
 \begin{equation}\label{measjpsiroth}
\Sigma_{\Psi}\left(Q_1, Q_2\right)- 25 \cdot\left(\Sigma_{\Psi}\left(Q_1, Q_2\right)\right)^2\;\le\;  \left| J_{\Psi}(Q_1, Q_2)\right|\;\le\;\Sigma_{\Psi}\left(Q_1, Q_2\right),
\end{equation}
where
 \begin{equation*}\label{measjpsirothbis}
\Sigma_{\Psi}\left(Q_1, Q_2\right)\;= \; \sum_{Q_1\le q< Q_2}\psi(q)\end{equation*}
and  where $\left|E\right|$ denotes the Lebesgue measure of a measurable subset $E\subset\R$. Of course, the first inequality in~\eqref{measjpsiroth} is only of interest when $\Sigma_{\Psi}\left(Q_1, Q_2\right)\le 1/5$, which restricts the range of relevant parameters $Q_1$ and $Q_2$. In the general case, the sieving process outlined in~\cite[p.25--28]{sprin} must be pushed further to get a finer lower bound. An alternative and more efficient approach is outlined in~\S\ref{derivationrothmaithm} below.\\%--- this analysis is left to the interested reader as it is not needed for the present purpose.\\

\noindent  Considering as above the set $\A_d(H)$ reduced modulo one to see if it intersects non-trivially $J_{\Psi}(Q_1, Q_2)$ accounts for the invariance modulo one in $\alpha$ of the Diophantine inequality~\eqref{rothineq}. This makes it equally natural to work instead with the set 
\begin{equation}\label{defrestrcalgnb}
\A'_d(H)=\A_d(H) \cap [0,1]
\end{equation} 
(the point is that the height of an algebraic number is not invariant under translation by an integer). The methods developed here are indeed flexible enough to deal with these two cases simultaneously. To simplify the notations, let from now on $\X_d(H)$ denote either of the sets $\A_d(H)$  or $\A'_d(H)$ . Here and throughout, given any integer $d\ge 0$, set furthermore
\begin{equation}\label{defepsikronecker}
 l(d)\;=
\begin{cases}
1& \textrm{if } d=2,\\
0&\textrm{otherwise}.
\end{cases}
\end{equation}
The following is a very particular, simplified case of the much more general statement established in Theorem~\ref{introeffeothbis} (see~\S\ref{statgenres})~:

\begin{thm}(Effective probabilistic estimates for the generalised Roth inequality)\label{introeffeoth}
Let $\Psi~: \R_+\rightarrow \R_+$ be an approximation function. Fix a degree $d\ge 2$ and integers $Q_2>Q_1\ge 1$.  Then, there exist  constants $H(d), H'(d),C(d), C'(d), C_1(\X_d), C_2(\X_d)>0$ such that upon defining for each $H\ge 1$ the error term
%Given $H\ge 2$ and constants $H(d), H'(d), C(d), C'(d)>0$ , define the error term 
\begin{equation*}\label{deferrortermroth}
E_{\X_d}(H)\;=\;
\begin{cases}
C(d)\cdot \sqrt{\left(\log H\right)^{l(d)}/ H} &\textrm{ when } \X_d=\A_d \textrm{ and } H\ge H(d);\\
C'(d)\cdot \left(\log H\right)^{ l(d)}/ H &\textrm{ when } \X_d=\A'_d \textrm{ and } H\ge H'(d),
\end{cases}
\end{equation*}
%Then, there exists %continuous strictly positive  probability density functions $\rho_{\X_d}~: [0,1]\rightarrow\R_+$, 
%constants $H(d), H'(d),C(d)$ and $C'(d)$ as above and also other constants $C_1(\X_d), C_2(\X_d)>0$ such that 
the following effective probabilistic estimates hold~:

\begin{itemize}
\item \textbf{Upper bound for the probability~:} let 
$\delta\in (0,1)$ and 
\begin{equation}\label{defthetapsi}
\Theta_{\Psi}(Q_1, Q_2)\; :=\; \min_{Q_1\le q< Q_2}2\cdot \frac{\Psi(q)}{q}\cdotp
\end{equation} 
Assume that $\Theta_{\Psi}(Q_1, Q_2)<C_1(\X_d)\cdot \delta$. Then, 
\begin{align*}\label{effconvdferoth}
&\frac{\#\left(J_{\Psi}(Q_1, Q_2)\cap \X_d(H)\right)}{\#\X_d(H)}\\ 
&\qquad\qquad \le\; C_2(\X_d)\cdot  \left|J_{\Psi}(Q_1, Q_2)\right|\cdot\left(1+\Theta_{\Psi}(Q_1, Q_2)\right)\cdot \left(\frac{1}{1-\delta}+\frac{E_{\X_d}(H)}{\Theta_{\Psi}(Q_1, Q_2)}\right).
\end{align*}

\item \textbf{Lower bound for the probability~:} let $\widetilde{J}_{\Psi}(Q_1, Q_2)$ be the closure of the complement in $[0,1]$ of the set $J_{\Psi}(Q_1, Q_2)$. Let $len^*\left(\widetilde{J}_{\Psi}(Q_1, Q_2)\right)$ denote the minimal length of the connected components of $\widetilde{J}_{\Psi}(Q_1, Q_2)$~:

\begin{itemize}
\item if $len^*\left(\widetilde{J}_{\Psi}(Q_1, Q_2)\right)=0$, then $J_{\Psi}(Q_1, Q_2)=[0,1]$ and in particular  $$\frac{\#\left(J_{\Psi}(Q_1, Q_2)\cap \X_d(H)\right)}{\#\X_d(H)}\;=\;1.$$
\item if $len^*\!\left(\widetilde{J}_{\Psi}(Q_1, Q_2)\right)>0$, then for any $\delta\in (0,1)$ such that $len^*\!\left(\widetilde{J}_{\Psi}(Q_1, Q_2)\right)< C_1(\X_d)\cdot \delta,$ 
\begin{align*}%\label{effdivroth}
1-&\frac{\#\left(J_{\Psi}(Q_1, Q_2)\cap \X_d(H)\right)}{\#\X_d(H)}\;\\
&\qquad \qquad \qquad \qquad \le C_2(\X_d)\cdot  \left|\widetilde{J}_{\Psi}(Q_1, Q_2)\right|\cdot\left(1+len^*\!\left(\widetilde{J}_{\Psi}(Q_1, Q_2)\right)\right)\\
&\qquad \qquad  \quad \qquad \qquad \qquad\qquad \qquad  \times \left(\frac{1}{1-\delta}+\frac{E_{\X_d}(H)}{len^*\!\left(\widetilde{J}_{\Psi}(Q_1, Q_2)\right)}\right).
\end{align*}
\end{itemize}
\end{itemize}

\noindent Furthermore, all the constants  are explicit and the quantity $len^*\!\left(\widetilde{J}_{\Psi}(Q_1, Q_2)\right)$ can be determined recursively in an effective way.
\end{thm}

\noindent To be precise, all quantities appearing above are effectively given in~\S\ref{derivationrothmaithm} (see~Equation~\eqref{liscst} onwards) although no attempt has been made to optimise them (the proofs make it clear that their values can be improved). This statement is therein seen as a particular case of a more general one where  the set of algebraic numbers of fixed degree and bounded height is endowed with a large class of discrete probability measures. \\

\noindent To illustrate Theorem~\ref{introeffeoth}, consider the case of algebraic numbers of degree $d=3$, assume that one works over the space $\X_3=\A_3$ and recall that the explicit values of the constants $C(3), C_1(\A_3), C_2(\A_3)$ and $H(3)$ are given in~\S\ref{derivationrothmaithm} . One can then retrieve the following estimates with for value $\delta=1/2$  upon using the measure estimates~\eqref{measjpsiroth} for the set $J_{\Psi}(Q_1, Q_2)$~:
\begin{itemize}
\item when  $\Psi(q)=1/(q\cdot \log^2 q)$, $Q_1=40\cdot C_2(\A_3)$ and $Q_2=50\cdot C_2(\A_3)$, the convergence case implies that it in the limit where $H$ tends to infinity, \emph{at most} 5\% of the algebraic cubic numbers ordered by naive heights allow for a solution to the Roth-type inequality~\eqref{rothineq} in an  irreducible fraction with a denominator between $Q_1$ and $Q_2$. When working at a finite distance, it is enough to choose $H=2500\cdot C_2(\A_3)^2/C(3)$ to ensure that this proportion should be less than 7\%;
\item when  $\Psi(q)=1/(q\cdot \log q)$, $Q_1=1000$ and $Q_2=5000$, the divergence case implies that it in the limit when $H$ tends to infinity, \emph{at least} 95\% of the algebraic cubic numbers ordered by naive heights allow for a solution to the Roth-type ine\-qua\-lity~\eqref{rothineq}  in an  irreducible fraction with a denominator between $Q_1$ and $Q_2$. When working at a finite distance, it is enough to choose $H=7\cdot 10^6$ to ensure that this proportion should be more than 92\%.
\end{itemize}

\noindent Theorem~\ref{introeffeoth} can also be interpreted in terms of the  concept of upper density naturally emerging  in the probabilistic regime. More precisely,  given an approximation function $\Psi$, define the upper density of the set of algebraic numbers admitting a large solution to the Roth-type inequality~\eqref{rothineq} with respect to $\Psi$ in the probabilistic regime  as 
\begin{equation}\label{defbanachproth}
\overline{d}(\Psi) \;=\; \limsup_{Q_1\rightarrow\infty} \;\;\limsup_{Q_2\rightarrow\infty} \;\;\limsup_{H\rightarrow \infty} \;\;\left(\frac{\#\left(J_{\Psi}(Q_1, Q_2)\cap \X_d(H)\right)}{\#\X_d(H)}\right).
\end{equation}
Similarly, define the analogous lower density $\underline{d}(\Psi)$ by replacing the upper limits above by lower limits. Finally, in the case that the upper and lower densities coincide, define the density $d(\Psi)$ by setting $$d(\Psi)\;=\;\underline{d}(\Psi)\;=\;\overline{d}(\Psi).$$

\begin{coro}[Khintchin--type result for the  density of the set of algebraic numbers satisfying a Roth-type inequality]\label{khincroorotheff}
Given any approximation function $\Psi$,  
\begin{equation*}\label{khincroorotheffbis}
d(\Psi) \;=\;
\begin{cases}
0 &\textrm{ if }\;\; \sum_{q=1}^{\infty}\Psi(q)<\infty;\\
1 &\textrm{ if } \;\; \sum_{q=1}^{\infty}\Psi(q)=\infty \textrm{ and if  } \;\Psi \textrm{ is monotonic}.
\end{cases}
\end{equation*}
In particular, the density $d(\Psi)$ is well--defined under the above assumptions.
\end{coro}

\noindent This result is related to a conjecture made by Waldschmidt --- see~\cite[\S~2.1]{ber2dod} and~\cite[Conjecture~2.12]{wald}. It claims that if $\Psi$ is a monotonic approximation function, then the Roth-type inequality~\eqref{rothineq} admits infinitely many solutions if, and only if, the series with general term $\Psi$ diverges. Thanks to work by Bugeaud and Moreira~\cite{bugmor}, Waldschmidt's conjecture is nevertheless known to fail in a very strong sense~: 
\begin{itemize}
\item it is proved in~\cite[Theorem~1.2]{bugmor} that for \emph{any} irrational number $\alpha$, there exists a nonincreasing function $\Psi$ (depending on $\alpha$) which is the general term of a \emph{divergent} series such that the inequality~\eqref{rothineq} admits  \emph{no} rational solution at all; 
\item in the opposite direction, it is established in~\cite[Theorem~1.1]{bugmor}  that if $\alpha$ is  assumed not to be badly approximable\textsuperscript{4}\let\thefootnote\relax\footnotetext{\textsuperscript{4} {in the sense that $\inf_{q\ge 1, p\in\Z}q\cdot \left|q\alpha-p\right|=0$.}}, then there exists an approximation function   $\Psi$ (depending on $\alpha$) which is the general term of a  \emph{convergent} series such that the inequality~\eqref{rothineq} admits  \emph{infinitely many} rational solutions. 
\end{itemize}
Since it is widely conjectured that algebraic numbers of degree at least 3 are not badly approximable, these two results provide a rather complete picture of the way Waldschmidt's conjecture fails.\\

\noindent Before the work~\cite{bugmor} by Bugeaud and Moreira, Beresnevich, Bernik and Dodson asked for a density version of Waldschmidt's conjecture in~\cite[\S2.1]{ber2dod}~: \begin{quote}\textit{Given an approximation function $\Psi$, how large is the set of $\Psi$--approximable algebraic numbers of degree $d$ (\textsuperscript{5})\let\thefootnote\relax\footnotetext{\textsuperscript{5} {That is, the set of algebraic numbers of degree $d$ for which the Roth--type inequality~\eqref{rothineq} admits infinitely many solutions.}}
 compared to the set of all algebraic numbers of degree $d$?  This question would give a partial  ‘density’ answer to
Waldschmidt’s conjecture.} \end{quote}The above Theorem~\ref{introeffeoth} and its Corollary~\ref{khincroorotheff}  provide a complete answer to this problem upon formalising the underlying concept of density as the one defined in~\eqref{defbanachproth}.

\subsection{A Reformulation of the Subspace Theorem.}

The  Main Problem stated in \S\ref{mainpbsec} above is tackled  in the full generality of the Subspace Theorem  from a geometric standpoint for which a projective reformulation of Theorem~\ref{subbasis} is more convenient. To this end, given a vector $\bm{x}\in \R^n\backslash\{\bm{0}\}$, denote by $\left[\bm{x}\right]  $ the line spanned by it seen as an element of the projective space $\Pj_{n-1}(\R)$ identified with the Grassmannian $G_{n,1}(\R)$. More generally, for each integer $1\le k\le n-1$, let $G_{n,k}(\R)$ denote the set of $k$--dimensional linear subspaces of $\R^n$. If $L\in \Pj_{n-1}(\R)$ is a projective point, let also $L^{\perp}$ be the subspace orthogonal to it seen as an element of the Grassmannian $G_{n, n-1}(\R)$. In the case that $L= \left[\bm{x}\right] $ for some $\bm{x}\in \R^n\backslash\{\bm{0}\}$, set for the sake of simplicity $L^{\perp}= \bm{x}^{\perp}$.\\

\noindent Fix two lines $L, \Lambda\in G_{n,1}(\R)$. The projective distance $\delta(L, \Lambda^{\perp})$ between the subspaces $L$ and $\Lambda^{\perp}$ is by definition the sine of the acute angle between the line $L\subset \R^n$ and the hyperplane $ \Lambda^{\perp}\subset \R^n$. When $L=\left[\bm{x}\right]$ and $\Lambda=\left[\bm{y}\right]$ for some $\bm{x}, \bm{y}\in \R^n\backslash\{\bm{0}\}$, set for simplicity $\delta(L, \Lambda^{\perp})=\delta(\bm{x}, \bm{y}^{\perp})$ in such a way that 
\begin{equation}\label{defprojdistan}
\delta(\bm{x}, \bm{y}^{\perp})\;=\;\frac{\left|\bm{x\cdotp y}\right|}{\left\|\bm{x}\right\|_2\cdotp \left\|\bm{y}\right\|_2}\cdotp
\end{equation}

\noindent The height $H(L)$ of a rational line $L\in \Pj_{n-1}(\Q)$ spanned by a nonzero integer point $\bm{q}\in\Z^n$ is defined as the Euclidean norm of the vector $\bm{q}/\gcd(\bm{q})$ (where $\gcd(\bm{q})$ is the gcd of the coordinates of $\bm{q}$). When $1\le k\le n-1$, the height $H(\Xi)$ of a $k$--dimensional rational subspace $\Xi$ (that is, of an element $\Xi$ in the Grassmannian $G_{n,k}(\Q)$) is defined in the usual way as the height of the projective point determined by the Pl\"ucker coordinates of this subspace (that is, by the coordinates of the wedge product of any of its rational bases --- see~\cite{schmidt0} for details). \\

\noindent It is known,  see~\cite[Corollary p.25]{cassels} and~\cite[Lemma~2]{laurbug}, that given a vector $\bm{q}\in\Z\backslash\left\{\bm{0}\right\}$, one always has that $H(\bm{q}^{\perp})\le \left\|\bm{q}\right\|_2$. Furthermore, equality holds whenever $\bm{q}$ is primitive (in the sense that $\gcd(\bm{q})=1$). This observation and the above definitions are easily seen to yield the following equivalent projective reformulation of Theorem~\ref{subbasis}. Here and throughout, the notation $\A$ refers to the set of \emph{real} algebraic numbers. Also, given algebraic vectors $ \bm{\alpha}_1, \dots, \bm{\alpha}_n\in \A^n\backslash\{\bm{0}\}$, each spanning a line denoted by $L_i$, recall that the \emph{orthognality defect} of these lines is the real defined as 
\begin{equation}\label{deforthodefct} 
\xi\left(L_1, \dots , L_n\right)\;=\; \frac{\left\|\bm{\alpha}_1 \wedge \dots. \wedge \bm{\alpha}_n\right\|_2}{ \left\|\bm{\alpha}_1\right\|_2\cdots \left\|\bm{\alpha}_n\right\|_2}  \;=\; \frac{\left|\det\left( \bm{\alpha}_1, \dots, \bm{\alpha}_n\right)\right|}{ \left\|\bm{\alpha}_1\right\|_2\cdots \left\|\bm{\alpha}_n\right\|_2}\cdotp 
\end{equation}
This ratio is strictly positive precisely when the  $ L_1, \dots, L_n$  are linearly independent. It is bounded above by 1 from Hadamard's Inequality, with the equality holding precisely when the lines are mutually orthogonal --- see~\cite[eq.~(13) p.49]{whitney} for further detail. 

\begin{thm}[Reformulation of the Subspace Theorem]\label{projsubthem}
Let $ L_1, \dots, L_n\in G_{1,n}(\A)$ be $n$ linearly independent algebraic lines in $\R^n$. Let $\varepsilon>0$. Then, the set of rational solutions $\Lambda\in G_{1,n}(\Q)$ to the inequality 
\begin{equation}\label{subspbis}
\prod_{i=1}^n \delta(L_i, \Lambda^{\perp})\;\le\;\frac{\xi\left(L_1, \dots , L_n\right)}{ H\left(\Lambda\right)^{n}
} \cdot \Psi\left(H\left(\Lambda\right)\right)
\end{equation}
lies in a finite union of proper rational subspaces of $\Q^n$ when $\Psi(H(\Lambda))=H(\Lambda)^{-\varepsilon}$. Here, $\xi\left(L_1, \dots , L_n\right)\in \left(0,1\right]$ denotes the orthogonality defect of the system  $ L_1, \dots, L_n$.
\end{thm}

%\noindent Recall the definition of the quantity $\xi\left(L_1, \dots , L_n\right) $~: if $ \bm{\alpha}_1, \dots, \bm{\alpha}_n\in \A^n\backslash\{\bm{0}\}$ are algebraic vectors such that the line $L_i$ is spanned by $\bm{\alpha}_i$ for each $i=1, \dots, n$, then 
%\begin{equation}\label{deforthodefct} 
%\xi\left(L_1, \dots , L_n\right)\;=\; \frac{\left\|\bm{\alpha}_1 \wedge \dots. \wedge \bm{\alpha}_n\right\|_2}{ \left\|\bm{\alpha}_1\right\|_2\cdots \left\|\bm{\alpha}_n\right\|_2}  \;=\; \frac{\left|\det\left( \bm{\alpha}_1, \dots, \bm{\alpha}_n\right)\right|}{ \left\|\bm{\alpha}_1\right\|_2\cdots \left\|\bm{\alpha}_n\right\|_2}\cdotp 
%\end{equation}
%This ratio is strictly positive precisely when the lines $ L_1, \dots, L_n$  are linearly independent. It is bounded above by 1 from Hadamard's Inequality, with the equality holding precisely when the lines are mutually orthogonal --- see~\cite[eq.~(13) p.49]{whitney} for further detail. 

\subsection{Heights and Degrees of Algebraic Vectors through Semialgebraic Parametrisations.}\label{heightsdegalgvec}

The analytic nature of the problem (relying on various volume calculations) makes it natural to define the heights and degrees of points in $\Pj_{n-1}(\R)$ from a real parametrisation of the projective space $\Pj_{n-1}(\R)$. This also allows for the consideration of a large class of height and degree functions. To achieve this, let $$  \bm{e}_0\;=•\; \left[1 : 0 : \dots :  0  \right], \quad   \bm{e}_1\;=•\; \left[0 : 1 : \dots : 0  \right],  \quad \; \dots \;  \quad ,  \quad \bm{e}_{n-1}\;=\; \left[0 : 0 : \dots :  1  \right]$$ be the canonical basis vectors of $\Pj_{n-1}(\R)$. Extending the definition~\eqref{defprojdistan}, let 
\begin{equation}\label{defprojdistanbis}
\delta(\bm{x}, \bm{y})\;=\;\frac{\left\|\bm{x}\wedge \bm{y}\right\|_2}{\left\|\bm{x}\right\|_2\cdotp \left\|\bm{y}\right\|_2}
\end{equation}
stand for the usual projective distance between  the lines $L=\left[ \bm{x}\right]$ and $ \Lambda=\left[ \bm{y}\right]$ in $ \Pj_{n-1}(\R)$. The height and degree functions under consideration are defined over the  dense open subset 
\begin{equation}\label{defprojdistanbis}
 \Pj_{n-1}^{\circ}(\R)\;=\; \left\{ L\in  \Pj_{n-1}(\R)\; : \; \delta(L, \left[  \bm{e}_0\right])\neq 1 \right\}.
\end{equation} 
Following the canonical identification of the projective space with  the sphere $\Sph^{n-1}(\R)$ quotiented by the relation of antipodal equivalence, this set is identified with the open upper hemi-sphere
\begin{equation*}\label{opeuphemsphe}
 \Sph^{n-1}_+(\R)\;=\; \left\{ \bm{x}=\left(x_0, \dots, x_{n-1}\right)\in  \Sph^{n-1}(\R)\; : \;  x_0>0 \right\}
\end{equation*} 
through the map
\begin{equation*}\label{mapid}
\bm{\iota}\; :\;   L\in   \Pj_{n-1}^{\circ}(\R)\;\;  \mapsto\;\;  \bm{\iota}(L) =\left(\iota_0(L), \dots, \iota_{n-1}(L) \right)\in  \Sph^{n-1}_+(\R)
\end{equation*} 
defined by taking the line $L$ to its point of intersection $ \bm{\iota}(L)$ with $ \Sph^{n-1}_+(\R)$. The class of admissible parametrisations considered throughout are those bi--Lipschitz semialgebraic $C^1$--diffeomorphisms mappings\textsuperscript{6}\let\thefootnote\relax\footnotetext{\textsuperscript{6} {Recall that a map is semialgebraic if so is its graph, and that a set is semialgebraic if it is a finite union of  sets defined by polynomial equalities and inequalities.}} defined over a subset $\mathfrak{S}$ of the open set  $ \Pj_{n-1}^{\circ}(\R)$ and taking values in a bounded subset of $\R^{n-1}$. Denoting by
\begin{equation}\label{mapalgparam}
\bm{\varphi}\; :\;   L\in   \mathfrak{S}\;\;  \mapsto\;\;  \bm{\varphi}(L) =\left(\varphi_1(L), \dots, \varphi_{n-1}(L) \right)
\end{equation} 
such a mapping, it is furthermore required that  an admissible parametrisation should meet the following two additional conditions~: 

\begin{itemize}
\item[(1)] (Algebraic property). There exist another semialgebraic map $R_k(\bm{X})$ defined over the field of  algebraic numbers, where $1\le k\le n-1$ and $\bm{X}=\left( X_1, \dots , X_{n-1}\right)$, such that for all $ L\in  \mathfrak{S}$ and all $1\le k\le n-1$, 
\begin{equation}\label{polynalgparam}
R_k( \bm{\varphi}(L), \iota_k(L))\;=\; 0.
\end{equation} 
\item[(2)] (Uniform directional regularity).  The directional derivative $\partial_{\bm{u}}\bm{\varphi}^{-1}$ of the inverse map $\bm{\varphi}^{-1}$ in any direction $\bm{u}\in \mathfrak{S}$ is uniformly bounded below in the sense that 
\begin{equation}\label{polynalgparambis}
\inf\left\{\partial_{\bm{u}} \bm{\varphi}^{-1}(\bm{x}) \; :\;  \bm{u}\in\mathfrak{S}, \;\bm{x}\in Im(\bm{\varphi})\right\}>0.
\end{equation} 
\end{itemize} 

\noindent The assumptions (1) and (2) appear very naturally in the context of the Subspace Theorem (see~\S\ref{sec4.1}). The canonical example of such an admissible semialgebraic parametrisation, which  subsumes the whole theory developed throughout, is the stereographic projection with respect to the "south pole" $(-1, 0, \dots, 0)$, namely the map
\begin{equation}\label{stereo} 
\bm{\sigma}\; :\;   L\in   \Pj_{n-1}^{\circ}(\R)\;\;  \mapsto\;\;  \bm{\sigma}(L)=\left(\sigma_1(L), \dots, \sigma_{n-1}(L) \right)\in  \B_{n-1}(\R) .
\end{equation} 
Here, $\B_{n-1}(\R) $ denotes the unit open Euclidean ball in $\R^{n-1}$, and the stereographic coordinates of the vector $\bm{\sigma}(L)$ are related to the coordinates of the point $ \bm{\iota}(L)$ through the algebraic identity
\begin{equation}\label{stereocoords} 
 \bm{\iota}(L)\;=\; \left(\frac{1-\left\|  \bm{\sigma}(L)\right\|_2^2}{1+\left\|  \bm{\sigma}(L)\right\|_2^2}, \;\frac{2\bm{\sigma}(L)}{1+\left\|  \bm{\sigma}(L)\right\|_2^2}\right).
\end{equation} 
It should be emphasized that all the counting and density results stated hereafter are independent of the choice of the base point $(-1, 0, \dots, 0)$ (instead of the opposite of any other canonical basis vector) when defining all the quantities between~\eqref{defprojdistanbis} and~\eqref{stereocoords}. \\

\noindent Another more basic example of admissible semialgebraic parametrisation is the one defined over any compact subset of the chart $\{x_k\neq 0\}$ in projective space\textsuperscript{7}\let\thefootnote\relax\footnotetext{\textsuperscript{7} {Note that the space $ \Pj_{n-1}(\R)$ can be covered by finitely many such compact sets.}}, where $0 \le k \le n-1$, by setting 
%\begin{equation}\label{deftauparam}
$\bm{\tau}(L)=(x_i/x_k)_{0\le i\le n-1, i\neq k}$ when $L=\left[x_0:\cdots, x_{n-1}\right]$ (\textsuperscript{8})\let\thefootnote\relax\footnotetext{\textsuperscript{8} {The two canonical choices between the parametrisations $\bm{\sigma}$ and $\bm{\tau}$ parallel the choice in~\S\ref{secrothintro} in the definition of the height of an algebraic number  either from the noncompact set $\A_d(H)$ reduced modulo one or from the compact set $\A'_d(H)$.}}.\\
%\end{equation}

\noindent Given an algebraic \emph{number} $\alpha\in\A$, denote its the degree by $\deg(\alpha)$ and  its (naive) height  by $\h(\alpha)$. If $\bm{\alpha}=\left(\alpha_1, \dots, \alpha_n\right)\in\A^{n}$ is an algebraic \emph{point}, the vectors $$\bdeg(\bm{\alpha})\;=\;(\deg(\alpha_1), \dots, \deg(\alpha_n))\;\; \;\textrm{and}\;\;\; \bh(\bm{\alpha})\;=\; \left(\h(\alpha_1), \dots, \h(\alpha_n)\right)$$ are referred to as its \emph{multi-degree} and its \emph{multi-height}, respectively. \\

\noindent  The theory developed throughout is flexible enough to deal with the multi-degrees and multi-heights of algebraic lines induced by the  choice of any admissible semialgebraic parametrisation as in~\eqref{mapalgparam}. With this in mind, given $L\in   \mathfrak{S}$ defined over the field of algebraic numbers, set
\begin{equation}\label{defmultidegheight}
\bdeg(L)  \;=\; \bdeg( \bm{\varphi}(L)) \qquad \textrm{and} \qquad \bh(L)\;=\;  \bh(\bm{\varphi}(L)).
\end{equation} 
These quantities are easily seen to be well--defined from the assumption~\eqref{polynalgparam}. For the sake of simplicity, they are from now on simply referred to as the \emph{degree} and the \emph{height} of the line $L$.\\

\noindent  A fundamental set of interest in the context of the Subspace Theorem is that of all those algebraic lines with fixed degree and bounded height. Given a vector of degrees $\bm{d}=\left(d_1, \dots, d_{n-1}\right)\in \N_{\ge 2}^{n-1}$ and a vector of heights $\bm{H}=\left(H_1, \dots, H_{n-1}\right)$, this is the set 
\begin{equation}\label{defboundegheightalflin}
P_{\bm{\varphi}}\!\left(\bm{d}, \bm{H}\right)\;=\; \left\{ L\in   \Pj_{n-1}^{\circ}(\A)\; :\; \bdeg(L)  =\bm{d} \quad \textrm{and} \quad \bh(L)\in\prod_{j=1}^{n-1} \left[0, H_j\right] \right\}.
\end{equation} 

\noindent Recalling that  $\A_d(H)$ denotes the set of real algebraic numbers of fixed degree $d\ge 1$ and (naive) height at most $H\ge 0$,  set also 
\begin{equation}\label{prodalg} 
\A_{{\bm{d}}}(\bm{H})= \prod_{j=1}^{n-1}\A_{d_j}(H_j).
\end{equation} 
It should then be clear that, by definition, 
\begin{equation}\label{pAdH}
\#P_{\bm{\varphi}}\!\left(\bm{d}, \bm{H}\right)\;=\; \# \left( \A_{\bm{d}}(\bm{H})\cap  \im(\bm{\varphi})\right).
\end{equation}

 \subsection{ Probabilistic Estimates for the Generalised Subspace   Inequality}\label{genprobestimsubs} 

\noindent  Let $\Psi~: \R_+\mapsto \R_+$ be an approximation function. The Main Problem is considered in the generality where the Subspace Inequality~\eqref{subspbis} is satisfied for any approximation function meeting a convergence condition (of which the original case $\Psi(x)=x^{-\varepsilon}$ is a particular example).\\

\noindent Let $1\le k\le n$. Say that a $k$-dimensional subspace $\Pi$ of $\Q^n$ (that is, an element of the Grassmannian $G_{k,n}(\Q)$) is a $\Psi$-solution to the Subspace Inequality~\eqref{subspbis} given algebraic lines $(L_1, \dots, L_n)\in G_{1,n}(\A)$ if~:
\begin{itemize}
\item there exists infinitely many rational lines $\Lambda\in G_{1,n}(\Q)$ in $\Pi$ satisfying the Subspace Inequality~\eqref{subspbis};
\item there exists $k$ linearly independent rational lines $\Lambda\in G_{1,n}(\Q)$ in $\Pi$ satisfying the Subspace Inequality~\eqref{subspbis}.
\end{itemize}
These stipulations allow for a greater generality upon working with a fixed value of $k$. Indeed, if the first were not satisfied, the finitely many solutions within $\Pi$ should be regarded as a finite set of one-dimensional solution subspaces, and if the second one were not satisfied, then a particular proper subspace of $\Pi$ should be regarded as the appropriate solution subspace. Note in particular that the first point rules out the particular case of Roth's Theorem dealt with separately in \S\ref{secrothintro} above.\\

\noindent  Let $1\le k<n$ and $Q_1, Q_2\ge 1$ be integers. The density estimates are concerned with the number of algebraic elements of fixed degree and bounded height in the set 
\begin{equation}\label{defjpsikq1q2}
J_\Psi\left(k; Q_1, Q_2\right).
\end{equation} 
This is defined as the set of all those lines $\left(L_1, \dots, L_n\right)\in \left(\Pj_{n-1}^{\circ}(\R)\right)^n$ for which there exists a $\Psi$-solution rational subspace $\Pi\in G_{k,n}(\Q)$ with (usual) height $H(\Pi)\ge Q_1$ for which there exist $k$ linearly independent rational lines $\Lambda\in \Pi\cap G_{1,n}(\Q)$ satisfying the Subspace Inequality~\eqref{subspbis} and obeying $H(\Lambda)\le Q_2$. \\

\noindent The Main Problem is tackled in the next theorem  by providing effective bounds for the proportion of algebraic lines with fixed degree and bounded height admitting a $\Psi$-solution to the Subspace Inequality~\eqref{subspbis}. In the above notations, it comes down to estimating the ratio $\#\left(P_{\varphi}\!\left(\bm{d}, \bm{H}\right) \cap J_\Psi\left(k; Q_1, Q_2\right)\right)/ \# P_{\varphi}\!\left(\bm{d}, \bm{H}\right)$, where the denominator can also be expressed as in~\eqref{pAdH}. % and~\eqref{crdAdH}. 

\begin{thm}(Effective Probabilistic Estimates for the Generalised Subspace   Inequality)\label{introeffesubs}
 Assume that $\bm{\varphi}$ is an admissible   parametrisation as in~\eqref{mapalgparam}. Fix a system of degrees $\bm{d}\in\N_{\ge 2}^{n-1}$ and a dimension $1\le k < n$ and denote by $\omega_k$  the volume of the $k$--dimensional unit Euclidean ball. Let also $\Psi$ be an approximation function such that $\Psi(Q)\ge Q^{-n}$ satisfying the convergence condition 
 \begin{equation}\label{convcondi}
 \sum_{Q=1}^\infty \frac{\Psi(Q)}{Q}\cdot\left(\log Q\right)^{n-1}\; <\; \infty.
 \end{equation}
 Then, there exists an effective constant $C_n(\bm{d}, \bm{\varphi})>0$ such that, given integers $Q_1\ge 1$ and  $Q_2\ge 2(Q_1/(k!\cdot \omega_k))^{1/k}$, it holds that  for any system of heights $\bm{H}=\left(H_1, \dots, H_{n-1}\right)$,
\begin{align}%\label{convcondi}
&\frac{\#\left(P_{\bm{\varphi}}\!\left(\bm{d}, \bm{H}\right) \cap J_\Psi\left(k; Q_1, Q_2\right)\right)}{\# P_{\varphi}\left(\bm{d}, \bm{H}\right)} \nonumber \\
&\le C_n(\bm{d}, \bm{\varphi})\cdot\left(\sum_{2(Q_1/(k!\cdot \omega_k))^{1/k}\le Q\le Q_2} \frac{\Psi(Q)}{Q}\cdot\left(\log Q\right)^{n-1}+Q_2^n\cdot \max_{1\le j<n}\frac{\left(\log H_j\right)^{l(d_j)}}{H_j}\right).\label{effboundsub}
 \end{align}
 In particular, under the convergence condition~\eqref{convcondi}, the  density of the set of algebraic vectors admitting a large rational  $\Psi$-solution to the  Generalised Subspace   Inequality~\eqref{subspbis}  in a fixed dimension $k$  vanishes in the probabilistic regime in the sense that 
  \begin{equation*}\label{banachdenssubsp}
 \limsup_{Q_1\rightarrow\infty}\;\;  \limsup_{Q_2\rightarrow\infty}\;\; \limsup_{ \min(\bm{H})\rightarrow\infty}\; \;\left(\frac{\#\left(P_{\varphi}\!\left(\bm{d}, \bm{H}\right) \cap J_\Psi\left(k; Q_1, Q_2\right)\right)}{\# P_{\bm{\varphi}}\!\left(\bm{d}, \bm{H}\right)}\right)\; =\; 0.
 \end{equation*}
Here, $ \min(\bm{H})=\min_{1\le j\le n-1}H_{j}$. 
%  Furthermore, one can choose $$C_n(\bm{d}, \bm{\varphi})=,$$ where $$blabla..$$ In the particular case where the semialgebraic mapping $\bm{\varphi}$ is the stereographic projection $\bm{\sigma}$ defined in~\eqref{stereo}, this reduces to $$C_n(\bm{d}, \bm{\sigma})=$$
 \end{thm}
% \noindent Under the convergence assumption, one thus obtains that 
% \begin{equation*}\label{banachdenssubsp}
% \limsup_{\underset{Q_2>2(Q_1/k!)^{1/k}}{Q_1\rightarrow\infty}}\;\;\; \limsup_{ \min(\bm{H})\rightarrow\infty}\; \left(\frac{\#\left(P_{\varphi}\left(\bm{d}, \bm{H}\right) \cap J_\Psi\left(k; Q_1, Q_2\right)\right)}{\# P_{\varphi}\left(\bm{d}, \bm{H}\right)}\right)\; =\; 0,
% \end{equation*}
% where $ \min(\bm{H})=\min_{1\le j\le n-1}H_{j}$ and where the first upper limit is taken as $Q_1$ tends to infinity when $Q_2$ is seen as any function of $Q_1$ taking (possibly infinite) values larger than $2(Q_1/k!)^{1/k}$. This can be interpreted as follows~: in the regime where the heights of the algebraic lines are larger than those imposed on the rational solutions, and when the convergence condition~\eqref{convcondi} is met, the Banach upper density of the set of large rational  $\Psi$-solutions to the Subspace Inequality~\eqref{subspbis} in a fixed dimension $k$ vanishes. 
\noindent An explicit value for the constant $C_n(\bm{d}, \bm{\varphi})$% and a sharp estimate for the cardinality of the set of reference $P_{\varphi}\left(\bm{d}, \bm{H}\right)$ 
is given in~\eqref{cndvarphi}, and this constant is further refined in the case of the stereographic parametrisation~\eqref{stereo} to a value given in~\eqref{cndsigma}. Here again, no attempt has been made to optimise its value although the proof makes it clear that it can be so.\\
 
 \noindent Estimate~\eqref{effboundsub} implies that for any fixed value of $Q_1\ge 1$,
  \begin{align}
& \limsup_{Q_2\rightarrow\infty}\;\;\; \limsup_{ \min(\bm{H})\rightarrow\infty}\; \left(\frac{\#\left(P_{\varphi}\!\left(\bm{d}, \bm{H}\right) \cap J_\Psi\left(k; Q_1, Q_2\right)\right)}{\# P_{\varphi}\!\left(\bm{d}, \bm{H}\right)}\right)\nonumber\\ 
&\qquad \qquad \qquad \qquad \qquad  \le \; C_n(\bm{d}, \bm{\varphi})\cdot\left(\sum_{Q\ge 2(Q_1/k!)^{1/k}} \frac{\Psi(Q)}{Q}\cdot\left(\log Q\right)^{n-1}\right).\label{banachdenssubspbis}
 \end{align}
This provides,  in the regime where the parameters $\min(\bm{H})$ and then $Q_2$ are taken asymptotically,  a quantitative upper bound on the density of algebraic points of fixed degrees for which a $\Psi$-solution to the Subspace Inequality~\eqref{subspbis} exists over $k$--dimensional subspaces of  heights larger than some fixed value $Q_1$.\\
  
 \noindent To illustrate numerically the effectiveness of the   bound~\eqref{effboundsub}, consider the stereographic projection in the case of $n=3$ projective lines spanned by   algebraic vectors, each of degrees  (defined from the stereographic parametrisation) $\bm{d}=(d_1,d_2, d_3)=(3,3,3)\; (=\bm{3})$. Choose the solution subspaces to be of maximal dimension $k=n-1=3$, and consider  the approximation function $\Psi=1/(\log Q)^4$ (which meets the convergence condition~\eqref{convcondi}). The formula~\eqref{cndsigma} in the case of the stereographic projection gives an explicit value for the constant $C_3(\bm{3}, \bm{\sigma})$. The following numerical values are then valid~:
 \begin{itemize}
 \item in the regime where $\min(\bm{H})$ and then $Q_2$ tend to infinity, inequality~\eqref{banachdenssubspbis} implies that in order for at most 5\% of the algebraic points to admit a solution in a rational subspace of height at least $Q_1$, it is enough to take $Q_1=60\cdot e^{C_3(\bm{3}, \bm{\sigma})}$. Conversely, when fixing the value $Q_1=51\cdot e^{C_3(\bm{3}, \bm{\sigma})}$,  at most 6\% of the algebraic points admit a $\Psi$-solution in a rational subspace of height at least $Q_1$;
 
 \item this can be made still more quantitative. To see this,  take for simplicity  $Q_2$ to infinity in the sum appearing in the first term in~\eqref{effboundsub} and take the same value of $Q_1=60\cdot e^{C_3(\bm{3}, \bm{\sigma})}$ as above. Then, in order for at most 8\% of the algebraic points to admit a $\Psi$-solution with norm at most $Q_2$ in a subspace of height at least $Q_1$, it is enough to impose that they should all be chosen with heights $H_1, H_2, H_3$ (defined from the stereographic parametrisation) all less than $150\cdot C_3(\bm{3}, \sigma)\cdot Q_2^2$.
 \end{itemize}

\section{Analytic Properties of Koleda's Density Function}

Let $d\ge 2$ be an integer. A key result at the foundation of the effective results stated in the introduction is due to Koleda~\cite{koleda}. It implies in particular that algebraic numbers of degree $d$  ordered by naive heights are not uniformly distributed in any compact interval of the real line. The goal in this liminary, technical section is to establish  quantitative versions of it which are needed subsequently.

\begin{thm}[Koleda, 2017]\label{koledadensitythm}
Let $d\ge 2$. Then, there exists a constant $K_1(d)>0$ and a continuous, strictly positively valued probability density function $\rho_d~: \R\rightarrow\R_{>0}$ such that for any subinterval $I\subset\R$ and any integer $H\ge 1$, 
\begin{equation*}
\left| \frac{\#\left(\A_d(H)\cap I\right)}{H^{d+1}}-\frac{1}{2\cdot\zeta(d+1)}\int_I\rho_d\right|\;\le\; K_1(d) \cdotp \frac{\left(\log H\right)^{l(d)}}{H},
\end{equation*}
where the quantity $l(d)$ is defined in~\eqref{defepsikronecker}. In this relation, the probability density function $\rho_d$ is explicitly given for any $x\in\R$ by 
\begin{equation}\label{koledadensity}
\rho_d(x)\;=\;  \int_{\Delta_d(x)}\left|\sum_{k=1}^{d}kp_kx^{k-1}\right|\cdot dp_1\cdots dp_d,
\end{equation}
where $\zeta$ is the Riemann function and where $\Delta_d(x)$ is the semialgebraic domain of integration
\begin{equation*}\label{koledadensity}
\Delta_d(x)\;=\; \left\{\left(p_1, \dots, p_d\right)\in \left[-1, 1\right]^d\;: \; \left|\sum_{k=1}^{d}p_kx^k\right|\le 1\right\}.
\end{equation*}
Moreover, the  function $\rho_d$ is even and satisfies for any $x\neq 0$ the identity
\begin{equation}\label{propkoledadensity}
x^2\cdot\rho_d(x)\;=\; \rho_d\left(\frac{1}{x}\right).
\end{equation}
\end{thm}

\noindent The  function $\rho_d$ defined in~\eqref{koledadensity} is from now on referred to as \emph{Koleda's density function}.

\begin{prop}[Properties of Koleda's probability density function]\label{propkolpdf}
Keep the above notations. Then,
\begin{itemize}
\item[1.] (\textbf{Explicit constant in the error term}.) An admissible value of the constant $K_1(d)$ in Theorem~\ref{koledadensitythm} is 
\begin{equation}\label{boundkoledpdf}
K_1(d)\;=\;  d\cdot\left(12 d^2\cdot 2^{d(d-2)}\cdot (d+1)^{d/2}+2^{d+2}\cdot p(d)\cdot \left(2 p(d)-1\right)^{s(d)}\right),
\end{equation}
where $$p(d)= 4^{\left(4^{d+1}-1\right)/3}\cdot d^{4^{d+1}}\quad \textrm{  and }\quad s(d)= 3\cdot\left(2d+3\right)^{2^{d+1}}\cdot d^{\left(16^{d+1}-1\right)\cdot \left\lceil \log d/ \log 2\right\rceil}.$$
\item[2.] (\textbf{Bounds on the density function}.) It holds that
\begin{equation}\label{boundkoledpdf} 
\rho_d(x)\;\le\; \frac{d(d+1)}{2} \qquad \textrm{for all }\qquad x\in\R
\end{equation}
and that 
\begin{equation}\label{boundkoledpdfbis}
\frac{1}{3d}\cdot\left(\frac{2}{ 3d}\right)^d \;\le\; \rho_d(x)  \qquad \textrm{for all }\qquad x\in\left[-1, 1\right].
\end{equation}
\item[3.] (\textbf{Lipschitz continuity of the density function}.) Set 
\begin{equation}\label{cstlipskoledpdf} 
K_2(d)\;=\; \frac{d(d+1)}{2}\cdot\left(4^{d}+(2\sqrt{d})^{d+3}\cdot \omega_{d-1}\right),
\end{equation}
where $\omega_{d-1}$ stands for the volume of the unit Euclidean ball in dimension $d-1$. Then, 
\begin{equation*}\label{lipskoledpdf}
\left|\rho_d(x)-\rho_d(y)\right|\;\le\; K_2(d)\cdot \left|x-y\right|\quad \textrm{whenever}\quad \left|x-y\right| \;\le\;\frac{1}{2^{d+1}d}
\end{equation*}
and $0\le x,y\le 1$.
\end{itemize}
\end{prop}

\begin{proof}
\noindent 1. This is a very particular case of the much more general statement given in Theorem~\ref{thm1.1} below. In the notations within, the above value of $K_1(d)$ follows  when taking the semialgebraic familiy $Z$ as the set of real intervals and then upon setting $n=1$, $s(Z)=2$ and $p(Z)=1$.\\

\noindent 2. Identity~\eqref{propkoledadensity} imply that $\rho_d$ attains its maximum over the interval $[-1,1]$. Thus, $$\max_{\R}\rho_d\;\le\;  \int_{[-1,1]^d}\left(\sum_{k=1}^{d}k\left|p_k\right|\right)\cdot dp_1\cdots dp_d\;\le\;\frac{d(d+1)}{2},$$ which proves the upper bound~\eqref{boundkoledpdf}. As for the lower bound~\eqref{boundkoledpdfbis}, the evenness of the map $\rho_d$ allows for restricting the variable $x$ to the interval $[0,1]$. Consider then the ball $B_{\infty}^{(d)}$ centered at the vector in $\R^d$ all of whose components are $2/(3d)$ with radius $1/(3d)$ with respect to the sup norm. Then, for all $\left(p_1, \dots, p_d\right)\in B_{\infty}^{(d)}$ and all $x\in[0,1]$, $$0\;\le\;\sum_{k=1}^{d} p_kx^k\;\le\; 1\qquad \textrm{and}\qquad \sum_{k=1}^{d}kp_kx^{k-1}\;\ge\; p_1\;\ge\; \frac{1}{3d}.$$ The left--hand side inequalities imply that for all $x\in[0,1]$, the set $B_{\infty}^{(d)}$ is contained in $ \Delta_d(x).$ From the right--hand side inequalities, one then deduces that $$\rho_d(x)\;=\;  \int_{\Delta_d(x)}\left|\sum_{k=1}^{d}kp_kx^{k-1}\right|\cdot dp_1\cdots dp_d \;\ge\; \frac{\V_d(B_{\infty}^{(d)})}{3d}\;=\; \frac{2^d}{(3d)^{d+1}},$$ where $\V_d(B_{\infty}^{(d)})$ denotes the Lebesgue measure of the set $B_{\infty}^{(d)}$. This concludes the proof of~\eqref{boundkoledpdfbis}.\\

\noindent 3. Let $x\in [0,1)$ and $\varepsilon>0$ be such that 
\begin{equation}\label{inegxepsi}
0\le x< x+\varepsilon\le 1.
 \end{equation}
 Then, 
\begin{align*}
\left|\rho_p(x+\varepsilon)-\rho_d(x)\right|\;&\le\; \int_{\mathcal{A}_d(x, \varepsilon)}\left| \;\left|\sum_{k=1}^{d}kp_k(x+\varepsilon)^{k-1}\right|- \left|\sum_{k=1}^{d}kp_kx^{k-1}\right|\;\right|\cdot dp_1\cdots dp_d\\
&+\; \int_{\mathcal{B}_d(x, \varepsilon)}\max\left\{ \left|\sum_{k=1}^{d}kp_k(x+\varepsilon)^{k-1}\right|; \left|\sum_{k=1}^{d}kp_kx^{k-1}\right|\right\}\cdot dp_1\cdots dp_d,
\end{align*}
where 
\begin{equation*}
\mathcal{A}_d(x, \varepsilon)\;=\; \Delta_d(x+\varepsilon)\cap \Delta_d(x) \;\subset\; [-1,1]^d
\end{equation*}
and
\begin{equation*}
 \mathcal{B}_d(x, \varepsilon)\;=\; \left(\Delta_d(x+\varepsilon)\cup \Delta_d(x)\right)\backslash\left(\Delta_d(x+\varepsilon)\cap \Delta_d(x)\right).
\end{equation*}
Since for all $\left(p_1, \dots, p_d\right)\in [-1, 1]^d$,  $$\left|\sum_{k=1}^{d}kp_k(x+\varepsilon)^{k-1} -\sum_{k=1}^{d}kp_kx^{k-1}\right|\;\le\; \varepsilon\sum_{k=1}^{d}2^k k\left|p_k\right|\;\le\; 2^{d-1}d(d+1)\varepsilon,$$ one obtains that 
\begin{equation}\label{integralAxepsi}
\int_{\mathcal{A}_d(x, \varepsilon)}\left| \;\left|\sum_{k=1}^{d}kp_k(x+\varepsilon)^{k-1}\right|- \left|\sum_{k=1}^{d}kp_kx^{k-1}\right|\;\right|\cdot dp_1\cdots dp_d\;\le\; 4^d\cdot\frac{d(d+1)}{2}\cdot \varepsilon.
\end{equation}

\noindent As for the integral over the domain $\mathcal{B}_d(x, \varepsilon)$, one first shows the inclusion 
\begin{equation}\label{inclusionBxeps}
\mathcal{B}_d(x, \varepsilon)\subset \mathcal{C}_d(x, \varepsilon),
 \end{equation}
 where
\begin{equation*}\label{inclusionBxepsbis}
\mathcal{C}_d(x, \varepsilon)\;=\;\left\{\left(p_1, \dots, p_d\right)\in [-1, 1]^d\;:\; 1-2^dd\varepsilon\;\le\;  \left|\sum_{k=1}^dp_kx^k\right|\;\le\; 1+2^dd\varepsilon\right\}.
 \end{equation*}
 To see this, it is enough to note that for all $\left(p_1, \dots, p_d\right)\in [-1, 1]^d$, $$\left|\sum_{k=1}^{d}p_k(x+\varepsilon)^{k-1} -\sum_{k=1}^{d}p_kx^{k-1}\right|\;\le\; \varepsilon\cdot\left(\sum_{k=1}^d2^k\left|p_k\right|\right)\;\le\; 2^dd\varepsilon.$$ Then, $$\Delta_d(x+\varepsilon)\cap \Delta_d(x)\supset  \left\{\left(p_1, \dots, p_d\right)\in [-1, 1]^d\;:\;  \left|\sum_{k=1}^dp_kx^k\right|\;\le\; 1-2^dd\varepsilon\right\}$$ and also $$\Delta_d(x+\varepsilon)\cup \Delta_d(x)\subset \left\{\left(p_1, \dots, p_d\right)\in [-1, 1]^d\;:\;  \left|\sum_{k=1}^dp_kx^k\right|\;\le\; 1+2^dd\varepsilon\right\},$$ whence~\eqref{inclusionBxeps}. \\
 
 \noindent Given $x\in [0,1]$ and $\varepsilon>0$, a necessary condition for the set $\mathcal{C}_d(x, \varepsilon)$ to be nonempty is that $\sum_{k=1}^d x^k\ge 1-2^dd\varepsilon$, and therefore that 
 \begin{equation}\label{necConempt}
 x\;\ge\; \frac{1-2^dd\varepsilon}{d}\;>\;0.
 \end{equation}
Furthermore, the inequalities defining $\mathcal{C}_d(x, \varepsilon)$ can be equivalently restated as
 \begin{equation*}
\dist \left(\bm{p}, \bm{u}(x)^{\perp}\right)\;\in\; \left[\frac{1-2^dd\varepsilon}{\sqrt{x^2+\dots+x^{2d}}}, \frac{1+2^dd\varepsilon}{\sqrt{x^2+\dots+x^{2d}}}\right] 
 \end{equation*}
 upon putting $\bm{p}= \left(p_1, \dots, p_d\right)$ and $\bm{u}(x)= \frac{\left(x, \dots, x^d\right)}{\sqrt{x^2+\dots+x^{2d}}}\cdotp$ Assuming that 
 \begin{equation}\label{boundvarepsik}
 \varepsilon\;\le\; \frac{1}{2^{d+1}d},
 \end{equation} 
 the set of all those $\bm{p}\in\R^d$ meeting this requirement defines the union of two slabs, each of thickness 
 \begin{equation*}
 \frac{2^{d+1}d\varepsilon}{\sqrt{x^2+\dots+x^{2d}}}\;\underset{\eqref{necConempt}}{\le}\;\frac{2^{d+1}d\varepsilon}{x} \;\underset{\eqref{necConempt}}{\le}\;\frac{2^{d+1}d^2\varepsilon}{1-2^dd\varepsilon}\;\underset{\eqref{boundvarepsik}}{\le}\;\ 2^{d+2}d^2\varepsilon.
 \end{equation*}
 The volume of the set $\mathcal{C}_d(x, \varepsilon)\subset [-1, 1]^d$ is then bounded above by the measure of intersection of these two slabs with the Euclidean ball centered at the origin with radius $\sqrt{d}$, which is easily seen to be at most $$2\left(\sqrt{d}\right)^{d-1}\cdot \omega_{d-1}\cdot 2^{d+2}d^2\varepsilon.$$ As a consequence, 
 \begin{align*}
& \int_{\mathcal{B}_d(x, \varepsilon)}\max\left\{ \left|\sum_{k=1}^{d}kp_k(x+\varepsilon)^{k-1}\right|; \left|\sum_{k=1}^{d}kp_kx^{k-1}\right|\right\}\cdot dp_1\cdots dp_d\\
&\underset{\eqref{inclusionBxeps}}{\le}\; \V_d\left(\mathcal{C}_d(x, \varepsilon)\right)\cdot\max_{\bm{p}\in [-1, 1]^d} \max\left\{ \left|\sum_{k=1}^{d}kp_k(x+\varepsilon)^{k-1}\right|; \left|\sum_{k=1}^{d}kp_kx^{k-1}\right|\right\}\\
& \underset{\eqref{inegxepsi}\&\eqref{inclusionBxeps}}{\le}\; d^{(d-1)/2}\cdot B_{d-1}\cdot 2^{d+3}d^2\varepsilon\cdot \frac{d(d+1)}{2}\\
&=\; (2\sqrt{d})^{d+3}\cdot B_{d-1}\cdot \frac{d(d+1)}{2}\cdot\varepsilon.
 \end{align*}
 Combining this upper bound with~\eqref{integralAxepsi} completes the proof of the statement.
\end{proof}

\noindent Within the context of Theorem~\ref{introeffeoth} stated in~\S~\ref{secrothintro}, the probability density functions  of interest are derived from Koleda's and correspond to the cases considered therein. More precisely, when $\X_d=\A_d$,  the density function is obtained by periodising Koleda's and is defined as
\begin{equation}\label{densityperiodi}
\xi_d\;:\; x\in [0,1]\;\mapsto\; \sum_{n\in\Z}\rho_d(x+n).
\end{equation}
It is easily seen to be continuous and well--defined from identity~\eqref{propkoledadensity}.  When $\X_d$ is the set $\A'_d$ (as defined from~\eqref{defrestrcalgnb}),  the relevant  density function is obtained by restricting  Koleda's to the interval $[0,1]$  and by normalising it. It is defined as 
\begin{equation}\label{densityrestri} 
\chi_d\;:\; x\in [0,1]\;\mapsto\; c_d\cdot \rho_d(x),\qquad \textrm{where}\qquad c_d\;=\; \left(\int_0^1\rho_d\right)^{-1}.
\end{equation}

\noindent Of fundamental importance to retrieve the effective estimates stated in  Theorem~\ref{introeffeoth} are quantitative estimates on the moduli of uniform continuity of the functions $\chi_d$ and $\xi_d$. To establish them, given a strictly positive  and continuous function $f: x\in [0,1]\rightarrow\R_{>0}$, let
\begin{equation}\label{etaunif}
\eta_f\;:\; \varepsilon>0\;\mapsto\; \eta_f(\varepsilon)\ge 0
\end{equation}
be the function giving the infimum of the inverses of its pointwise moduli of continui\-ty in the following precise sense~:  $\eta_f$ is the  function such that for every $\varepsilon>0$, the real number $\eta_f(\varepsilon)$ is the infimum over all $\eta>0$ meeting the property  that  for any   $ x,y\in [0,1]$, it holds that $\left|f(x)-f(y)\right|\le \eta$ whenever $\left|x-y\right|\le \varepsilon$. Clearly, the map $\eta_f$ is well--defined and tends to $0$ with $\varepsilon$ from the assumption that $f$ is continuous, hence uniformly continuous, over the interval $[0,1]$. Another quantity of importance is the normalised level set of the function $\eta_f$ defined as
\begin{equation}\label{nuunif}
\nu_f\;:\; \delta>0\;\mapsto\; \sup\; \eta_f^{-1}\left(\right[0, \delta\cdot \min f \left]\right).
\end{equation}

\noindent  The functions $\eta_f$ and $\nu_f$ are estimated in the following two statements when $f$ is  successively taken as the probability density functions $\chi_d$ and $\xi_d$.

\begin{prop}[Uniform continuity properties of the periodised Koleda's density]\label{modunifchi}
Keep the same notations as in Theorem~\ref{koledadensitythm} and Proposition~\ref{propkolpdf}. Let $\xi_d$ be the density function defined in~\eqref{densityperiodi}. Then, its uniform pointwise modulus of continuity function $\eta_{\xi_d}$ and the corresponding level set function $\nu_{\xi_d}$ defined in~\eqref{etaunif} and~\eqref{nuunif}, respectively,  satisfy the following bounds for any $\varepsilon>0$ and $\delta>0$~:
\begin{equation*}\label{boundetanuxid}
\eta_{\xi_d}(\varepsilon) \; \le \; K_3(d)\cdot \varepsilon \qquad \textrm{and} \qquad \nu_{\xi_d}(\delta)\;\ge\;K_4(d)\cdot \delta.
\end{equation*}
Here,
\begin{equation}\label{K3}
K_3(d)\;=\;  K_2(d)\cdot \left(1+2\zeta(4)\right)\;+\;3 \zeta(3) \cdot d(d+1)
\end{equation}
with the constant $K_2(d)$ defined in~\eqref{cstlipskoledpdf} and 
\begin{equation}\label{K4}
K_4(d)\;=\;  \frac{1}{3d\cdot K_3(d)}\cdot \left(\frac{2}{3d}\right)^d\cdotp
\end{equation}
\end{prop}

\begin{proof}
Let $0\le x, y<1$. Upon setting $\tau=x-y$, one has 
\begin{align*}
\left|\xi_d(x)-\xi_d(y)\right|\;&\underset{\eqref{propkoledadensity}}{=}\;\left|\rho_d(x)-\rho_d(y)\right|+2\sum_{n=1}^{\infty}\left| \frac{\rho_d\left(\frac{1}{x+n}\right)}{(x+n)^2}- \frac{\rho_d\left(\frac{1}{y+n}\right)}{(y+n)^2}\right|\\
&\le\; \left|\rho_d(x)-\rho_d(y)\right|+2\sum_{n=1}^{\infty}\frac{1}{n^2}\cdot\left|\rho_d\left(\frac{1}{x+n}\right)-\left(\frac{x+n}{y+n}\right)^2\cdot \rho_d\left(\frac{1}{y+n}\right)\right|\\
&\le\; \left|\rho_d(x)-\rho_d(y)\right|+2\sum_{n=1}^{\infty}\frac{1}{n^2}\cdot\left|\rho_d\left(\frac{1}{x+n}\right)- \rho_d\left(\frac{1}{y+n}\right)\right|\\
&\qquad \qquad  +2\sum_{n=1}^{\infty}\frac{1}{n^2}\cdot\left|\frac{2\tau}{y+n}+\frac{\tau^2}{(y+n)^2} \right|\cdot \rho_d\left(\frac{1}{y+n}\right)
\end{align*}
Since $$\left|\frac{1}{x+n}-\frac{1}{y+n}\right|\;\le\; \frac{\left|x-y\right|}{n^2},$$ it follows that
\begin{equation*}
\left|\xi_d(x)-\xi_d(y)\right|\;\le\;  \left|\rho_d(x)-\rho_d(y)\right|+2\sum_{n=1}^{\infty}\frac{1}{n^2}\cdot \eta_{\rho_d}\left(\frac{\left|\tau\right|}{n^2}\right)+6\left|\tau\right|\cdot\left(\sum_{n=1}^{\infty}\frac{1}{n^3}\right)\cdot\max_{0\le t\le 1}\rho_d(t).
\end{equation*}
Under the assumption that $\left|\tau\right|\le 1/(2^{d+1}d)$, Points 2 \& 3 in Proposition~\ref{propkolpdf} yield that 
\begin{equation*}
\left|\xi_d(x)-\xi_d(y)\right|\;\le\;  \left(K_2(d)\cdot\left(1+2\zeta(4)\right)+3\zeta(3)\cdot d(d+1)\right)\cdot\left|x-y\right|.
\end{equation*}
Since, from  Point 2 in Proposition~\ref{propkolpdf} again,  it holds that
\begin{equation*}
\min_\R \xi_d\; \ge \; \min_{0\le t\le 1} \rho_d(t)\;\ge\; \frac{1}{3d}\cdot \left(\frac{2}{3d}\right)^d,
\end{equation*}
the proof is complete.
\end{proof}

\begin{prop}[Uniform continuity properties of the restricted Koleda's density]\label{modunifxi}
Keep the same notations as in Theorem~\ref{koledadensitythm} and Proposition~\ref{propkolpdf}. Let $\chi_d$ be the density function defined in~\eqref{densityrestri}. Then, its uniform pointwise modulus of continuity function $\eta_{\chi_d}$ and the corresponding level set function $\nu_{\chi_d}$ defined in~\eqref{etaunif} and~\eqref{nuunif}, respectively,  satisfy the following bounds for any $\varepsilon>0$ and $\delta>0$~:
\begin{equation*}\label{boundetanuxid}
\eta_{\chi_d}(\varepsilon) \; \le \; c_d\cdot  K_2(d)\cdot \varepsilon \qquad \textrm{and} \qquad \nu_{\chi_d}(\delta)\;\ge\;K_6(d)\cdot \delta.
\end{equation*}
Here, %
%\begin{equation*}\label{K5}
%K_5(d)\;=\; 2^{d-1}\cdot d(d+1)+2\cdot d^{3+(d+1)/2}(d+1)\cdot V_{d-1} \qquad (???)
%\end{equation*}
$c_d$ is the constant defined in~\eqref{densityrestri}, $K_2(d)$ the constant defined in~\eqref{cstlipskoledpdf} and  
\begin{equation}\label{K6}
K_6(d)\;=\; \frac{1}{3d\cdot c_d\cdot K_2(d)}\cdot \left(\frac{2}{3d}\right)^d\cdotp
\end{equation}
\end{prop}

\begin{proof}
These are immediate consequences of Points 2 \& 3 in Proposition~\ref{propkolpdf}.
\end{proof}

\section{General Counting Measures in Roth's Theorem}

\subsection{Statement of the General Result}\label{statgenres} 
 
\noindent Let $\left(\X(H)\right)_{H\ge 1}$  be a sequence of finite subsets of $\R$ (it should be emphasized that each $\X(H)$ need not coincide with the more specific subset $\X_d(H)\in \left\{\A_d(H), \A'_d(H)\right\}$ defined in the Introduction). Let then $\left(\mu_H\right)_{H\ge 1}$ be the sequence of probability measures on the unit interval $\left[0,1\right]$ defined as
\begin{equation}\label{defmuH}
\mu_H\;=\; \frac{1}{\#\X(H)}\sum_{\alpha\in \X(H)\!\!\!\!\!\pmod 1}\delta_{\alpha},
\end{equation}
where $\delta_{\alpha}$ denotes the Dirac mass at $\alpha$. Assume that $\left(\mu_H\right)_{H\ge 1}$ converges  weakly to a probability measure $\mu$ admitting   a strictly positive and continuous density function $\rho~: [0,1]\rightarrow \R_{>0}$. Assume also  that there exists a uniform error term sequence $E~:\N\rightarrow\R_{\ge 0}$ in the sense  that for all $H\ge 1$,
\begin{equation}\label{deferrortermmeasu}
\sup\;\left|\mu_H(I)- \int_I\rho \right|\;\le\; E(H),
\end{equation}
where the supremum is taken over all subintervals $I\subset [0,1]$.  \\

\noindent Theorem~\ref{introeffeoth} is derived in \S\ref{derivationrothmaithm} below from the following   more general statement. It gives quantitative probabilistic estimates in terms of the  uniform pointwise modulus of continuity function $\eta_\rho$ defined in~\eqref{etaunif} and of the corresponding level set function $\nu_\rho$ defined in~\eqref{nuunif}.

\begin{thm}(Effective probabilistic estimates for general counting measures in Roth's Theorem)\label{introeffeothbis}
Let $\Psi~: \R_{>0}\rightarrow \R_{>0}$ be an approximation function and let $\left(\mu_H\right)_{H\ge 1}$ be the sequence of probability measures defined in~\eqref{defmuH} assumed to satisfy property~\eqref{deferrortermmeasu}. Fix a degree $d\ge 2$, integers $Q_2>Q_1\ge 1$ and recall the definition of the set $J_{\Psi}(Q_1, Q_2)$ in~\eqref{jpsiroth}.

\begin{itemize}
\item \textbf{Upper bound for the probability~:} assume that
\begin{equation*}
\Theta_\Psi(Q_1, Q_2)\;:=\;\min_{Q_1\le q<Q_2}\frac{2\Psi(q)}{q}\;\;<\;\;\nu_{\rho}(\delta)\cdotp
\end{equation*} 
Then, 
\begin{align*}\label{effconvroth}
\frac{\#\left(J_{\Psi}(Q_1, Q_2)\cap \X(H)\right)}{\#\X(H)}\; &\le\; 4\left(\int_{J_{\Psi}(Q_1, Q_2)}\rho\right) \;+\; S_{\Psi, \rho}^{(\delta)}(H, Q_1, Q_2)\\
&\quad + \; 4 \left|J_{\Psi}(Q_1, Q_2)\right|\cdot\eta_\rho\left(  \Theta_\Psi(Q_1, Q_2)\right),
\end{align*}
where 
\begin{align*}
&S_{\Psi, \rho}^{(\delta)}(H, Q_1, Q_2)\;=\\ 
&\left(\frac{8\delta}{1-\delta}+\frac{E(H)}{\Theta_\Psi(Q_1, Q_2)\cdot \min \rho}\right) \cdot \left(\left|J_{\Psi}(Q_1, Q_2)\right|\cdot\eta_\rho\left(  \Theta_\Psi(Q_1, Q_2)\right)+\int_{J_{\Psi}(Q_1, Q_2)}\rho\right).
\end{align*} 
%and where $\left|A\right|$ denotes the Lebesgue measure of a measurable subset $A\subset\R$.

\item \textbf{Lower bound for the probability~:}  assume that $\delta\in (0,1)$ is such that $$0<len^*\left(\widetilde{J}_{\Psi}(Q_1, Q_2)\right)< \nu_{\rho}(\delta),$$ where $\widetilde{J}_{\Psi}(Q_1, Q_2)$ denotes the closure of the complement in $[0,1]$ of the set $J_{\Psi}(Q_1, Q_2)$, and where $len^*\left(\widetilde{J}_{\Psi}(Q_1, Q_2)\right)$ is the minimal length of its connected components. Then,  
\begin{align*}%\label{effconvrothbis}
1-\frac{\#\left(J_{\Psi}(Q_1, Q_2)\cap \X(H)\right)}{\#\X(H)}\; &\le\; 4\left(\int_{\widetilde{J}_{\Psi}(Q_1, Q_2)}\rho\right) \;+\; T_{\Psi, \rho}^{(\delta)}(H, Q_1, Q_2)\\
&\quad + \; 4 \left|\widetilde{J}_{\Psi/2}(Q_1, Q_2)\right|\cdot\eta_\rho\left( len^*\left(\widetilde{J}_{\Psi}(Q_1, Q_2)\right)\right)
\end{align*}
with
\begin{align*}
T_{\Psi, \rho}^{(\delta)}(H, Q_1, Q_2)\;=&\\ 
&\left(\frac{8\delta}{1-\delta}+\frac{E(H)}{len^*\left(\widetilde{J}_{\Psi}(Q_1, Q_2)\right)\cdot \min \rho}\right) \\
&\;\times \left(\left|\widetilde{J}_{\Psi}(Q_1, Q_2)\right|\cdot\eta_\rho\left( len^*\left(\widetilde{J}_{\Psi}(Q_1, Q_2)\right)\right)+\int_{\widetilde{J}_{\Psi}(Q_1, Q_2)}\rho\right).
\end{align*} 
Furthermore, if $len^*\left(\widetilde{J}_{\Psi}(Q_1, Q_2)\right)=0$, then the set $J_{\Psi}(Q_1, Q_2)$ coincides with the full interval $[0,1]$.
\end{itemize}
\end{thm}

\noindent The main ingredient in the proof of Theorem~\ref{introeffeothbis} is the following Proposition where, there and throughout this section, all notations and assumptions introduced in the above statement are kept. It is also convenient to set  $B(x, r)=[x-r, x+r]$ when $x\in\R$ and $r>0$. %Recall also notation $B^*(x,r)$ introduced in~\eqref{defballunit}.

\begin{prop}\label{prop1}
Let $H\ge 1$ be an integer and let $\delta\in (0,1)$, $\gamma\in [0,1)$ and $\varepsilon\in \left(0, \nu_\rho(\delta)/2\right)$ be reals such that $(1+\gamma)\varepsilon<1/2$. Assume that $\Omega$ is a measurable subset of $[0,1]$ and that $g~: [0,1]\rightarrow [0,1]$ is a continuous map such that 
\begin{equation}\label{propgx}
\left|g(\alpha)-\alpha\right|\le\gamma\varepsilon \qquad \textrm{for all}\qquad \alpha\in [0,1].
\end{equation} 
Then,
\begin{align*}
& \left(1-\gamma\right)\cdot\left(\int_\Omega\rho\right)- \left(\frac{2\delta}{1-\delta}+\frac{E(H)}{\varepsilon\cdot \min \rho}\right)\cdot \left(2\left|\Omega\right|\cdot\eta_\rho\left((1+\gamma)\cdot\varepsilon\right)+\int_\Omega\rho\right)\\
&\qquad\qquad +\gamma\cdot\frac{2\delta}{1-\delta}\cdot\left(\int_\Omega\rho-  \left|\Omega\right|\cdot\eta_\rho\left((1+\gamma)\cdot\varepsilon\right)\right)\\
& \le\; \int_0^1\frac{\left|B^*(g(\alpha), \varepsilon)\cap \Omega\right|}{\left|B^*(g(\alpha), \varepsilon)\right|}\cdot d\mu_H(\alpha) \\
& \le\;  2\left(1+\gamma\right)\cdot\left(\int_\Omega\rho\right)+2\left(\frac{2\delta (1+\gamma)}{1-\delta}+\frac{E(H)}{\varepsilon\cdot \min \rho}\right)\cdot \left(\left|\Omega\right|\cdot\eta_\rho\left((1+\gamma)\cdot\varepsilon\right)+\int_\Omega\rho\right)\\
&\qquad \qquad +2(1+\gamma)\cdot \left|\Omega\right|\cdot\eta_\rho\left((1+\gamma)\cdot\varepsilon\right).
\end{align*}
Here, given $\alpha\in\R$ and $r>0$, the set $B^*(\alpha,r)$ is defined in~\eqref{defballunit}.
\end{prop}

\noindent The proof of Proposition~\ref{prop1} relies on an auxiliary statement~: 

\begin{lem}\label{lemma 1.1}
Let $H\ge 1$ be an integer and let $\delta\in (0,1)$ be a real. Then, for any subinterval $I\subset [0,1]$ such that 
\begin{equation}\label{conditintervud}
0<\left|I\right|<\nu_\rho(\delta),
\end{equation}
it holds that
\begin{equation*}
\left|\frac{1}{\left|I\right|}\cdot\int_I\frac{d\mu_H(x)}{\rho(x)}-1\right|\;\le\;\frac{2\delta}{1-\delta}+\frac{E(H)}{\left|I\right|\cdot \min \rho}\cdotp
\end{equation*}
\end{lem}

\begin{proof}
Since 
\begin{align*}
\frac{\mu_H(I)}{\max_I\rho}\;\le\; \int_I\frac{d\mu_H(x)}{\rho(x)}\;\le\;\frac{\mu_H(I)}{\min_I\rho},
\end{align*}
it follows from the definition of the uniform error term sequence $E$ that 
\begin{align}\label{inegfrr}
\frac{1}{\max_I\rho}\cdot \left(\int_I\rho - E(H)\right)\;\le\; \int_I\frac{d\mu_H(x)}{\rho(x)} \;\le\; \frac{1}{\min_I\rho}\cdot \left(\int_I\rho + E(H)\right).
\end{align}
Let then $x^+\in I$ and $x^-\in I$ be points realizing the maximum and minimum, respectively, of $\rho$ over the interval $I$. Under  assumption~\eqref{conditintervud}, one has 
\begin{equation}\label{condidibs}
\eta_\rho(|I|)\le \delta\cdot\underset{[0,1]}{\min}\;\rho,
\end{equation} 
yielding
\begin{align*}
\frac{1-\delta}{1+\delta}\;\underset{\eqref{condidibs}}{\le}\; \frac{\rho(x^+)-\eta_\rho(|I|)}{\rho(x^+)+\eta_\rho(|I|)}\;&\le\; \frac{1}{\rho(x^+)}\cdot\frac{1}{|I|}\cdot \int_I\rho\\
&\le\; \frac{1}{\rho(x^-)}\cdot\frac{1}{|I|}\cdot \int_I\rho\;\le\;  \frac{\rho(x^-)+\eta_\rho(|I|)}{\rho(x^-)-\eta_\rho(|I|)}\;\underset{\eqref{condidibs}}{\le}\;\frac{1+\delta}{1-\delta}\cdotp
\end{align*}
Inequalities~\eqref{inegfrr} then imply that 
\begin{align*}
\frac{1-\delta}{1+\delta}- \frac{E(H)}{|I|\cdot \min\rho}\;\le\; \frac{1}{|I|}\cdot \int_I\frac{d\mu_H(x)}{\rho(x)} \;\le\; \frac{1+\delta}{1-\delta}+ \frac{E(H)}{|I|\cdot \min\rho},
\end{align*}
which concludes the proof.
\end{proof}

\begin{proof}[Proof of Proposition~\ref{prop1} from Lemma~\ref{lemma 1.1}]
Given $x, \alpha\in [0,1]$ such that assumption~\eqref{propgx} holds for $\alpha$, it is elementary to check the implications
\begin{equation}\label{incluballs}
\alpha \in B(x, (1-\gamma)\varepsilon)\qquad\Longrightarrow\qquad x\in B(g(\alpha), \varepsilon)\qquad\Longrightarrow\qquad \alpha\in B(x, (1+\gamma)\varepsilon).
\end{equation}
In other words, denoting by $\chi_A$ the characteristic function of a subset $A\subset\R$, one obtains under the same assumptions that 
\begin{equation*}
\chi_{B(x, (1-\gamma)\varepsilon)}(\alpha)\;\le\; \chi_{B(g(\alpha), \varepsilon)}(x)\;\le\; \chi_{B(x, (1+\gamma)\varepsilon)}(\alpha).
\end{equation*} 
Since $\rho$ takes strictly positive values, one has 
\begin{align*}
&\int\frac{\left|B(g(\alpha), \varepsilon)\cap \Omega\right|}{\left|B^*(g(\alpha), \varepsilon)\right|}\cdot d\mu_H(\alpha)\\
&\qquad\qquad =\; \int\frac{1}{\rho(\alpha)\cdot \left|B^*(g(\alpha), \varepsilon)\right|}\cdot\left(\int_{\Omega\cap B(g(\alpha), \varepsilon)}\left(\rho(\alpha)-\rho(x)+\rho(x)\right)\cdot dx\right)\cdot d\mu_H(\alpha).
\end{align*}
This implies that
\begin{align*}
&\left|\int\frac{\left|B(g(\alpha), \varepsilon)\cap \Omega\right|}{\left|B^*(g(\alpha), \varepsilon)\right|}\cdot d\mu_H(\alpha)-\int\frac{1}{\rho(\alpha)}\cdot\frac{1}{ \left|B^*(g(\alpha), \varepsilon)\right|}\cdot\left(\int_{\Omega\cap B(g(\alpha), \varepsilon)}\rho\right)\cdot d\mu_H(\alpha)\right|\\
&\;  \underset{\eqref{incluballs}}{\le}\; \eta_\rho\left((1+\gamma)\varepsilon\right) \cdot\int\frac{\left|B(g(\alpha), \varepsilon)\cap \Omega\right|}{\left|B^*(g(\alpha), \varepsilon)\right|}\cdot \frac{1}{\rho(\alpha)}\cdot\mu_H(\alpha)\\
& \le\; \eta_\rho\left((1+\gamma)\varepsilon\right) \cdot \int_{\Omega}\left(\int \frac{\chi_{B(g(\alpha), \varepsilon)}(x)}{\varepsilon\cdot \rho(\alpha)}\cdot d\mu_H(\alpha)\right)\cdot dx \\%\;\; \textrm{(from the Fubini-Tonelli Theorem)}\\
& \underset{\eqref{incluballs}}{\le}\;\; \eta_\rho\left((1+\gamma)\varepsilon\right) \cdot \int_{\Omega}\left(\int \frac{\chi_{B(x, (1+\gamma)\varepsilon))}(\alpha)}{\varepsilon\cdot \rho(\alpha)}\cdot d\mu_H(\alpha)\right)\cdot dx. 
\end{align*}
It is then an easy consequence of Lemma~\ref{lemma 1.1} that 
\begin{align*}
&\left|\int\frac{\left|B(g(\alpha), \varepsilon)\cap \Omega\right|}{\left|B^*(g(\alpha), \varepsilon)\right|}\cdot d\mu_H(\alpha)-\int\frac{1}{\rho(\alpha)}\cdot\frac{1}{  \left|B^*(g(\alpha), \varepsilon)\right|}\cdot\left(\int_{\Omega\cap B(g(\alpha), \varepsilon)}\rho\right)\cdot d\mu_H(\alpha)\right|\\
&\le\;  2\cdot\eta_\rho\left((1+\gamma)\varepsilon\right) \cdot \left|\Omega\right|\cdot (1+\gamma)\cdot\left(1+\frac{2\delta}{1-\delta}+\frac{E(H)}{(1+\gamma)\varepsilon\cdot \min \rho}\right).
\end{align*}
The same argument based on Lemma~\ref{lemma 1.1} %and on the Fubini--Tonelli Theorem 
also yields that 
\begin{align*}
\int\frac{1}{\rho(\alpha)}\cdot\frac{1}{\left|B^*(g(\alpha), \varepsilon)\right|}\cdot&\left(\int_{\Omega\cap B(g(\alpha), \varepsilon)}\rho\right)\cdot d\mu_H(\alpha)\\
&  \le\; 2(1+\gamma)\cdot \left(1+\frac{2\delta}{1-\delta}+\frac{E(H)}{(1+\gamma)\varepsilon\cdot \min \rho}\right)\cdot \left(\int_\Omega\rho\right).
\end{align*}
Finally, when used with the first implication in~\eqref{incluballs}, it leads one to the lower bound
\begin{align*}
\int\frac{1}{\rho(\alpha)}\cdot\frac{1}{\left|B^*(g(\alpha), \varepsilon)\right|}\cdot&\left(\int_{\Omega\cap B(g(\alpha), \varepsilon)}\rho\right)\cdot d\mu_H(\alpha)\\
&  \ge\; (1-\gamma)\cdot \left(1-\frac{2\delta}{1-\delta}-\frac{E(H)}{(1-\gamma)\varepsilon\cdot \min \rho}\right)\cdot \left(\int_\Omega\rho\right).
\end{align*}
The statement then follows upon applying the triangle inequality.
\end{proof}

\noindent The following corollary to Proposition~\ref{prop1}  is the key to the proof of Theorem~\ref{introeffeothbis}. Ge\-ne\-ra\-li\-sing the notation therein, given a set $\Omega\subset [0,1]$ defined as a finite union of closed intervals, denote by $len^*(\Omega)$ the minimal  length  of  its connected components.

\begin{coro}\label{lemma 1.2}
Assume that $\Omega\subset [0,1]$ is a finite union of closed intervals such that $len^*(\Omega)>0$. Let $H\ge 1$ be an integer and let $\delta\in (0,1)$ be a real such that $len^*(\Omega)<\nu_{\rho}(\delta)$. Then, 
\begin{align*}
\mu_H(\Omega)\; \le\;&4\left(\int_\Omega\rho\right)\\
&+\left(\frac{8\delta}{1-\delta}+\frac{E(H)}{len^*(\Omega)\cdot \min \rho}\right) \cdot \left(\left|\Omega\right|\cdot\eta_\rho\left(  len^*(\Omega)\right)+\int_\Omega\rho\right) \\
& +4 \left|\Omega\right|\cdot\eta_\rho\left(  len^*(\Omega)\right).
\end{align*}
\end{coro}

\begin{proof}
Set 
\begin{equation}\label{choiceepsi}
\varepsilon\;=\;\frac{len^*(\Omega)}{2}
\end{equation} 
and let $\gamma\in [0,1)$. It is easy to construct a piecewise affine and  continuous map $g~:[0,1]\mapsto [0,1]$ such that for any $\alpha\in\Omega$, $\left|\alpha-g(\alpha)\right|\le \gamma\varepsilon$ and $\left|B(g(\alpha),\varepsilon)\cap\Omega\right|\ge (1+\gamma)\varepsilon$.  Then, upon noticing that
\begin{align*}
\mu_H(\Omega)\;&\le\;\int_\Omega\frac{2}{1+\gamma}\cdot\frac{\left|B(g(\alpha),\varepsilon)\cap\Omega\right|}{2\varepsilon}\cdot\textrm{d}\mu_H(\alpha)\\
&\le\; \frac{2}{1+\gamma}\cdot\left( \int_0^1\frac{\left|B^*(g(\alpha), \varepsilon)\cap \Omega\right|}{\left|B^*(g(\alpha), \varepsilon)\right|}\cdot d\mu_H(\alpha)\right),
\end{align*}
it follows from Proposition~\ref{prop1} that 
\begin{align*}
\mu_H(\Omega)\;&\underset{\eqref{choiceepsi}}{\le}\;4\left(\int_\Omega\rho\right)\\
&+2\cdot \left(\frac{4\delta}{1-\delta}+\frac{E(H)}{len^*(\Omega)\cdot \min \rho\cdot (1+\gamma)}\right) \cdot \left(\left|\Omega\right|\cdot\eta_\rho\left(\frac{1+\gamma}{2}\cdot len^*(\Omega)\right)+\int_\Omega\rho\right) \\
& +4 \left|\Omega\right|\cdot\eta_\rho\left(\frac{1+\gamma}{2}\cdot len^*(\Omega)\right).
\end{align*}
Letting $\gamma\rightarrow 1^-$ concludes the proof.
\end{proof}

\begin{proof}[Completion of the proof of Theorem~\ref{introeffeothbis}]
\noindent \sloppy The upper bound  follows immediately from  Corollary~\ref{lemma 1.2} upon noticing that, in the notations of the theorem, $len^*\left(J_{\Psi}(Q_1, Q_2)\right)\ge \Theta_{\Psi}(Q_1, Q_2)$. \\ 

\noindent  As for the lower bound, note first that $len^*\!\left(\widetilde{J}_{\Psi}(Q_1, Q_2)\right)=0$ precisely when  the finite union of closed intervals $\widetilde{J}_{\Psi}(Q_1, Q_2)$ is %either 
the empty set (given that  $J_{\Psi}(Q_1, Q_2)$ is itself a finite union of closed intervals). %or a finite union of points. In both cases, this implies   that the closure of its complement, which is the set  $J_{\Psi}(Q_1, Q_2)$ given as a finite union of intervals, has full measure in $[0,1]$. 
Then,  $J_{\Psi}(Q_1, Q_2)=[0,1]$ and the claim follows in this case.\\

\noindent Assume now that  $0<len^*\!\left(\widetilde{J}_{\Psi}(Q_1, Q_2)\right)<\nu_\rho(\delta)$. The conclusion then also follows from Corollary~\ref{lemma 1.2} applied to the set $\widetilde{J}_{\Psi}(Q_1, Q_2)$ upon noticing the trivial inequality $$\frac{\#\left(J_{\Psi}(Q_1, Q_2)\cap \X(H)\right)}{\#\X(H)}+ \frac{\#\left(\widetilde{J}_{\Psi}(Q_1, Q_2)\cap \X(H)\right)}{\#\X(H)}\;\ge\; 1.$$
\end{proof}

\subsection{Deduction of Theorem~\ref{introeffeoth}}\label{derivationrothmaithm}

The goal in this section is to first show  that the uniform error estimate~\eqref{deferrortermmeasu} is met for the two discrete probability measures
\begin{equation*}\label{lambdaH}
\lambda_H\;=\; \frac{1}{\#\A_d(H)}\sum_{\alpha\in \A_d(H)\!\!\!\!\!\pmod 1}\delta_{\alpha}\quad \textrm{and} \quad \lambda'_H\;=\; \frac{1}{\#\left(\A_d(H)\cap [0,1]\right)}\sum_{\alpha\in \A_d(H)\cap [0,1]}\delta_{\alpha}
\end{equation*} 
respectively attached to the continuous probability density functions $\xi_d$  defined in~\eqref{densityperiodi} and
%\begin{equation*}\label{lambdaHprime}
%\lambda'_H\;=\; \frac{1}{\#\left(\A_d(H)\cap [0,1]\right)}\sum_{\alpha\in \A_d(H)\cap [0,1]}\delta_{\alpha}
%\end{equation*}
%attached to the continuous probability density function 
$\chi_d$ defined in~\eqref{densityrestri}. Theorem~\ref{introeffeoth}  and  Corollary~\ref{khincroorotheff} are then derived from this property.

\begin{lem}[Uniform error term estimate for the discrete probability measure $\lambda'_H$]\label{errlambdprime} Given an integer $H\ge 2$, let $$ y_d(H)\;=\; 2c_d\cdot K_1(d)\cdot \zeta(d+1)\cdot\frac{\left(\log H\right)^{l(d)}}{H}\cdotp $$ Then, under the assumption that $H$ is chosen so that  $y_d(H)\le 1/2$, assumption~\eqref{deferrortermmeasu} is satisfied for the discrete measure $\lambda'_H$ with the error term $$E_{\A'_d}(H)\; =\; 5y_d(H).$$ Here,   $K_1(d)$ and  $c_d$ are the constants defined in~\eqref{boundkoledpdf} and~\eqref{densityrestri}, respectively. %In particular,  $$E_{\A'_d}(H)\; \le\; 4x_d(H)\qquad \textrm{as soon as}\qquad x_d(H)\le 1.$$ 
\end{lem}

\begin{proof}
Given an integer $H\ge 1$,  Theorem~\ref{koledadensitythm} implies the existence of reals $\kappa, \kappa'\in [-1,1]$ such that 
\begin{equation*}
 \frac{\#\left(\A_d(H)\cap I\right)}{\#\left(\A_d(H)\cap [0,1]\right)}\;=\; c_d\cdot\left(\int_I\rho_d\right)\cdot\frac{1\;+\;\kappa\cdot \frac{2\cdot\zeta(d+1)\cdot K_1(d)}{\int_I\rho_d}\cdot \frac{\left(\log H\right)^{l(d)}}{H}}{1\;+\;\kappa'\cdot \left(2c_d\cdot\zeta(d+1)\cdot K_1(d)\right)\cdot \frac{\left(\log H\right)^{l(d)}}{H}}\cdotp
\end{equation*}
Upon using the inequality $1/(1+y)\le 1-y+2y^2$ valid for all $\left|x\right|\le 1/2$, elementary manipulations yield
\begin{equation*}
\left| \frac{\#\left(\A_d(H)\cap I\right)}{\#\left(\A_d(H)\cap [0,1]\right)}-\int_I\chi_d\right|\;\le\;   2y_d(H)\cdot\left(1+2y_d(H)+y_d(H)^2\right),
\end{equation*}
which quantity is indeed less that $5y_d(H)$ when $y_d(H)\le 1/2$.
\end{proof}

\noindent The proof in the case of the measure $\lambda_H$ is more subtle~:

\begin{lem}[Uniform error term estimate for the discrete probability measure $\lambda_H$]\label{errlambdprime2}  Given an integer $H\ge 2$, let $$ z_d(H)\;=\; K_1(d)\cdot \frac{\left(\log H\right)^{l(d)}}{H}\cdotp $$ Then, under the assumption that $H$ is chosen so that  $z_d(H)\le 1/2$, assumption~\eqref{deferrortermmeasu} is satisfied for the discrete measure $\lambda_H$  with the error term $$E_{\A_d}(H)\; =\; 24\cdot d\cdot\sqrt{z_d(H)}.$$
\end{lem}

\begin{proof}
Consider the discrete measure on $\R$ defined as $$\gamma_H\;=\; \frac{1}{\#\A_{\bm{d}}(H)}\cdot \sum_{\alpha\in\A_d(H)}\delta_\alpha.$$ It is such that  $\lambda_H(I)=\gamma_H(I+\Z)$ for any interval $I\subset [0,1)$. Furthermore, an argument similar to the one used in the proof of Lemma~\ref{errlambdprime} reveals that for any interval $J\subset\R$,  
\begin{equation}\label{gammasigmaIJ}
\left|\gamma_H(J)-\frac{1}{2\cdot \zeta(d+1)}\cdot\left(\int_J\rho_d\right)\right|\;\le\; 2z_d(H)\cdot\left(1+\frac{3}{2}\cdot z_d(H)+z_d(H)^2\right)\;\le\; 4\cdot z_d(H)
\end{equation}
under the assumption that $z_d(H)\le 1/2$. \\

\noindent Fix an interval $I\subset [0,1)$, introduce a cutoff bound $N\ge 1$ and set for the sake of simplicity of notation $\left\llbracket N\right\rrbracket= \Z\cap [-N, N]$. Then, $$\gamma_H(I+\left\llbracket N\right\rrbracket)\;\le \; \gamma_H(I+\Z)\;=\; \lambda_H(I)\;\le\; \gamma_H\left(I+\left\llbracket N\right\rrbracket\right)+\gamma_H\left(\R\backslash [-N, N]\right)$$ and relations~\eqref{gammasigmaIJ} leads one to the estimates
\begin{align}
&\int_I\left(\sum_{|n|\le N}\rho_d(x+n)\right)\cdot\frac{\textrm{d}x}{2\cdot\zeta(d+1)}\;-\;4(2N+1)\cdot z_d(H)\nonumber \\
&\le \lambda_H(I)\label{affaeffa}\\
&\le\; \int_I\left(\sum_{|n|\le N}\rho_d(x+n)\right)\cdot\frac{\textrm{d}x}{2\cdot \zeta(d+1)}\;+\;4(2N+1)\cdot z_d(H)\;+\;\gamma_H\left(\R\backslash [-N, N]\right). \nonumber
\end{align}
The last term  can be bounded as follows~: 
\begin{align}\label{afagafae}
\gamma_H\left(\R\backslash [-N, N]\right)\;&\underset{\eqref{gammasigmaIJ}}{\le}\; \frac{1}{2\cdot \zeta(d+1)}\cdot\left(\int_{-\infty}^{-N}\rho_d+ \int_{N}^{\infty}\rho_d\right)+8\cdot z_d(H)\nonumber \\
&\underset{\eqref{propkoledadensity}}{\le}\; 2\cdot\left(\int_N^{\infty}\frac{\rho_d(1/x)}{x^2}\cdot \frac{\textrm{d}x}{2\cdot\zeta(d+1)}\right)+ 8\cdot z_d(H)\nonumber \\
& \underset{\eqref{boundkoledpdf}}{\le}\;\frac{d(d+1)}{\zeta(d+1)}\cdot \frac{1}{N}+8\cdot z_d(H).
\end{align}
Also, 
\begin{align}\label{ladaz2}
\int_I\left(\sum_{|n|\le N}\rho_d(x+n)\right)\cdot\textrm{d}x\;&=\;  \int_I\left(\sum_{n\in\Z}\rho_d(x+n)\right)\cdot\textrm{d}x-\int_I\left(\sum_{|n|> N}\rho_d(x+n)\right)\cdot\textrm{d}x\nonumber \\
&\ge \; \int_I\left(\sum_{n\in\Z}\rho_d(x+n)\right)\cdot\textrm{d}x-2 \cdot\left(\int_{N+1}^\infty\rho_d\right)\nonumber \\
& \underset{\eqref{propkoledadensity}\&\eqref{boundkoledpdf}}{\ge}\;  \int_I\left(\sum_{n\in\Z}\rho_d(x+n)\right)\cdot\textrm{d}x-2d(d+1)\cdot\frac{1}{N+1}\cdotp
\end{align}
Combining inequalities~\eqref{affaeffa} with the bounds~\eqref{afagafae} and~\eqref{ladaz2}, one then obtains
$$\left|\lambda_H(I)-  \int_I\left(\sum_{n\in\Z}\rho_d(x+n)\right)\cdot\frac{\textrm{d}x}{2\cdot \zeta(d+1)}\right|\;\le\; 4(2N+1)\cdot z_d(H)+2\cdot \frac{d(d+1)}{N}+8\cdot z_d(H).$$  The right--hand side is optimized when $$N\;=\;\frac{1}{2}\cdot \sqrt{\frac{d(d+1)}{z_d(H)}},$$ in which case it is bounded above by $$8\cdot\sqrt{d(d+1)}\sqrt{z_d(H)}+12\cdot z_d(H)\;\le\; 24\cdot d\cdot z_d(H)$$ under the assumption that $z_d(H)\le 1/2$.
\end{proof}

\begin{proof}[Completion of the proof of Theorem~\ref{introeffeoth}] $\quad$\\

\noindent \emph{Part~I~: Effective determination of the constants.}  The estimates in the statement follow immediately from Theorem~\ref{introeffeothbis} and from the above Lemmata~\ref{errlambdprime2} and~\ref{errlambdprime} on the error term functions upon explicitating admissible values for the various constants therein.\\

\noindent To this end, consider first the case where $\X_d=\A_d$. From Lemma~\ref{errlambdprime2}, one can take 
\begin{equation}\label{liscst}
C(d)=24d\cdot\sqrt{K_1(d)} \quad \textrm{and}\quad H(d)=\min\left\{H\ge 2\; :\; 2K_1(d)\cdot \frac{\left( \log H\right)^{l(d)}}{H}\le 1 \right\},
\end{equation} 
where $K_1(d)$ is defined in~\eqref{boundkoledpdf}. In the case where $\X_d=\A'_d$, from Lemma~\ref{errlambdprime}, one can take $$C'(d)=10c_d\cdot K_1(d)\cdot\zeta(d+1)$$ and  $$ H'(d)=\min\left\{H\ge 2\, :\, 2c_d\cdot K_1(d)\cdot\zeta(d+1)\cdot \frac{\left( \log H\right)^{l(d)}}{H}\le 1 \right\},$$ where $c_d$ is defined in~\eqref{densityrestri}.\\

\noindent In both cases $\X_d=\A_d$ and $\X_d=\A'_d$, one can take 
\begin{equation*}
C_1(\X_d)\;=\; 
\begin{cases}
K_4(d)&\textrm{ if } \X_d=\A_d;\\
K_6(d)&\textrm{ if } \X_d=\A'_d,
\end{cases}
\end{equation*}
where the constants $K_4(d)$ and $K_6(d)$ are defined in~\eqref{K4} and~\eqref{K6}, respectively. Furthermore, one can take $$C_2(\X_d)\;=\; 8\cdot \max\left\{\max_{[0,1]}\rho_{\X_d}, \;\eta(\X_d), \;\frac{\max\left\{\max_{[0,1]}\rho_{\X_d}, \; \eta(\X_d)\right\}}{\min_{[0,1]}\rho_{\X_d}} \right\}, $$ where from Propositions~\ref{modunifchi} and~\ref{modunifxi}, 
\begin{equation*}
\eta(\X_d)\;=\; 
\begin{cases}
K_3(d) &\textrm{ if } \X_d=\A_d;\\
c_d\cdot K_2(d)&\textrm{ if } \X_d=\A'_d,
\end{cases}
\end{equation*}
with the constants $K_2(d)$ and $K_3(d)$  defined in~\eqref{cstlipskoledpdf} and~\eqref{K3}, respectively. Also, 
\begin{equation*}
\rho_{\X_d}\;=\; 
\begin{cases}
\xi_d &\textrm{ if } \X_d=\A_d;\\
\chi_d&\textrm{ if } \X_d=\A'_d,
\end{cases}
\end{equation*}
which functions are defined in~\eqref{densityperiodi} and~\eqref{densityrestri}, respectively.\\

\noindent The  above estimates can be further simplified upon utilising the inequalities 
\begin{equation}\label{boundcd} 
\frac{2}{d(d+1)}\;\underset{\eqref{boundkoledpdf}}{\le}\;  c_d\;\underset{\eqref{boundkoledpdfbis}}{\le}\; 3d\cdot \left(\frac{3d}{2}\right)^{d}
\end{equation} 
and also $$ \frac{1}{3d}\cdot \left(\frac{2}{3d}\right)^d\;\underset{\eqref{boundkoledpdfbis}}{\le}\; \min_{[0,1]} \rho_d \;\underset{\eqref{densityperiodi}}{\le}\; \min_\R \xi_d\;\le\; \max_\R\xi_d\;\underset{\eqref{propkoledadensity}, \eqref{boundkoledpdf}\&\eqref{boundkoledpdfbis}}{\le}\; \left(1+\zeta(2)\right)\cdot d(d+1)$$ as well as 
\begin{align*} 
\frac{d+1}{6}\cdot \left(\frac{2}{3d}\right)^d &\underset{\eqref{boundkoledpdfbis} \&\eqref{boundcd}}{\le} c_d\cdot \min_{[0,1]} \rho_d\\
&\underset{\eqref{densityrestri}}{\le} \min_\R \chi_d \le\max_\R\chi_d\underset{\eqref{densityrestri}}{\le} c_d\cdot \max_{[0,1]} \rho_d \underset{ \eqref{boundkoledpdf}\&\eqref{boundcd}}{\le} d(d+1)\cdot \left(\frac{3d}{2}\right)^{d+1}.
\end{align*}

\noindent \emph{Part~II~: Recursive determination of the quantity $len^*\!\left(\widetilde{J}_{\Psi}(Q_1, Q_2)\right)$.} 
\noindent %Under the assumption that the map $q\mapsto \Psi(q)/q$ is nonincreasing, the quantities $\Theta_{\Psi}(Q_1, Q_2)$ and $\lambda_{\Psi}(Q_1, Q_2)$ appearing in the statement can be estimated exactly. Indeed, it is then clearly the case that $$\Theta_{\Psi}(Q_1, Q_2)\;=\;2\cdot\frac{\Psi(Q_2-1)}{Q_2-1}\cdotp$$ As for $\lambda_{\Psi}(Q_1, Q_2)$, it can be given as a recursive formula. To this end, d
Denote by $$\mathcal{F}\left(Q_2-1\right)\;=\; \left\{\frac{p_0}{q_0}\;<\; \frac{p_1}{q_1}\;<\; \dots\;<\; \frac{p_{N(Q_2)}}{q_{N(Q_2)}}\right\}$$ the Farey sequence of order $Q_2-1$. (Recall that this is the set of all ordered irreducible fractions in $[0,1]$ with denominators at most $Q_2-1$ and that the cardinality of this set is $N(Q_2)+1$, where  $N(Q_2)=\sum_{1\le k\le Q_2-1}\varphi(k)$ with $\varphi$ being Euler's totient function.) Starting from the values $p_0/q_0=0/1$ and $p_1/q_1=1/(Q_2-1)$, the denominator $q_k$ ($k\ge 2$) is determined recursively from the identity $$q_k\;=\; \left\lfloor \frac{Q_2-1+q_{k-2}}{q_{k-1}} \right\rfloor\cdot q_{k-1}-q_{k-2} $$ (see~\cite[Equation~(1.1)]{BZ}). The numerator $p_k$ is then characterised by the classical relation $p_kq_{k-1}-p_{k-1}q_k=1$. \\ 

\noindent Given an irreducible fraction $p/q$, consider then the modified approximating function 
\begin{equation*}
\tau\left(\frac{p}{q}\right)\;=\; \max_{\underset{s\ge 1, r\ge 0}{s|q}}\left\{\frac{\psi(s)}{s}\; :\; qr=ps\right\}
\end{equation*}
(which reduces to  $\Psi$ when it is nonincreasing). Let also 
\begin{align*}
& l^+(Q_1, Q_2, i)\;=\; \\
&\qquad \qquad  \max_{i<j\le N(Q_2)}\left\{  \frac{p_j}{q_j}-\frac{p_i}{q_i}+  \tau\left(\frac{p_j}{q_j}\right)\; :\;    \frac{p_j}{q_j} - \tau\left(\frac{p_j}{q_j}\right)>  \frac{p_i}{q_i} + \tau\left(\frac{p_j}{q_j}\right) \quad \textrm{and}\quad q_j\ge Q_1\right\}
\end{align*} 
when $i<N(Q_2)$ and
\begin{align*}
& l^-(Q_1, Q_2, i)\;=\; \\
&\qquad \qquad  \max_{0\le j <i}\left\{  \frac{p_i}{q_i}-\frac{p_j}{q_j}+  \tau\left(\frac{p_j}{q_j}\right)\; :\;    \frac{p_j}{q_j} + \tau\left(\frac{p_j}{q_j}\right)<  \frac{p_i}{q_i} - \tau\left(\frac{p_j}{q_j}\right) \quad \textrm{and}\quad q_j\ge Q_1\right\}
\end{align*} 
when $i>0$.\\

\noindent Let then $\left\{i_1<\dots<i_{K(Q_1, Q_2)} \right\}$ be an enumeration of the indices $i$ for which $q_i\ge Q_1$. Upon setting $[x]_+=\max\{x, 0\}$ for $x\in\R$,  it is easy to see that $$len^*\!\left(\widetilde{J}_{\Psi}(Q_1, Q_2)\right)= \sum_{k=0}^{K(Q_1, Q_2)+1} \left[\left(\frac{p_{i_{k+1}}}{q_{i_{k+1}}}-\frac{p_{i_{k}}}{q_{i_{k}}}\right)-\left(l^+(Q_1, Q_2, i_k)+l^-(Q_1, Q_2, i_{k+1})\right)\right]_+,$$  \sloppy where the boundary values are taken as $(i_0, i_{K(Q_1, Q_2)+1})=(0, N(Q_2))$ and $\left(l^+(Q_1, Q_2, 0), \; l^-(Q_1, Q_2, N(Q_2)\right)=(0,0)$.\\

\noindent This completes the proof of Theorem~\ref{introeffeoth}.
\end{proof}

\begin{proof}[Deduction of Corollary~\ref{khincroorotheff}]
\noindent In the convergence case,  the quantity $\Theta_{\Psi}(Q_1, Q_2)$ defined in~\eqref{defthetapsi} vanishes when $Q_2$ tends to infinity for any fixed value of $Q_1$. Also, the upper bound in~\eqref{measjpsiroth} shows that the sequence $\left(U_{\Psi}(Q_1)\right)_{Q_1\ge 1}$ defined by setting $U_{\Psi}(Q_1)=\limsup_{Q_2\rightarrow\infty}\left|J_{\Psi}(Q_1, Q_2)\right|$ vanishes at infinity. These two observations are enough to derive in this case the sought conclusion from the upper bound for the probability  established in Theorem~\ref{introeffeoth}.\\

\noindent  In the divergence case, under the  assumption of the monotonicity of the function $\Psi$,  the measure of the complementary set $J_{\Psi}^c(Q_1, Q_2)$ vanishes as $Q_2$ tends to infinity for any fixed value of $Q_1$~: this is an equivalent restatement of the classical Khintchin's theorem in Diophantine approximation. Clearly, the same limit property then follows for the sequence of minimal lengths $\left(len^*\!\left(\widetilde{J}_{\Psi}(Q_1, Q_2)\right)\right)_{Q_2>Q_1}$. This observation immediately implies the sought conclusion from the lower bound for the probability  established in Theorem~\ref{introeffeoth}. 
\end{proof}

\section[Distribution of the Solutions to the Generalised Subspace Inequality]{Distribution of the Solutions to the Subspace Inequality}

\subsection{Algebraic Points in  Semialgebraic Families}\label{sec4.1}

A \emph{semialgebraic family} is a map $Z$ from a semi-algebraic set $\Omega\subset \R^m$ to the set of all subsets  \sloppy of $\R^n$ (with $m,n\ge 0$ integers) such that the graph set $Z^*=\left\{(\bm{t}, \bm{x})\in\Omega\times \R^n\;:\; \bm{x}\in Z(\bm{t})\right\}$ is semialgebraic. Given $\bm{t}\in\Omega$, call $Z(\bm{t})$   the \emph{fibre} over $\bm{t}$. Throughout, the integers 
\begin{equation}\label{defpsZ}
s(Z) \qquad \qquad \textrm{and} \qquad \qquad p(Z)
\end{equation}
denote respectively the number of polynomial equalities and inequalities used in a given representation of the graph set $Z^*$,  and the maximal degrees of the polynomials involved in such a representation.\\

\noindent Theorem~\ref{introeffesubs} is derived from the theorem below established in the following section. To state it, given  an $n$-tuple of positive integers $\bm{d}=\left(d_1, \dots, d_n\right)\in \left(\N_{\ge 2}\right)^n$,  set for the sake of simplicity of notations $d\bm{x}=dx_1\dots dx_n$, $$\zeta(\bm{d}+\bm{1})\;=\; \prod_{i=1}^{n}\zeta(d_i+1)\quad \textrm{and} \quad \bm{\rho}_{\bm{d}}(\bm{x})\;=\;\prod_{i=1}^{n}\rho_{d_i}(x_i)\quad \textrm{for all}\quad \bm{x}=(x_1, \dots, x_n)\in\R^n,$$ where $\rho_{d_i}$ is Koleda's density function~\eqref{koledadensity}. When $\bm{H}=(H_1, \dots, H_n)$ is an $n$--tuple of integers, it will also be convenient to set $\bm{H}^{\bm{d}}=\prod_{1\le i\le n}H_i^{d_i}$.  Finally, recall the definition of the set $\A_{\bm{d}}(\bm{H})$ in~\eqref{prodalg}.

\begin{thm}[Counting algebraic points in  semialgebraic families]\label{thm1.1}
Let $m, n\ge 1$ be integers and let $\bm{d}=\left(d_1, \dots, d_n\right)\in \left(\N_{\ge 2}\right)^n$ be an $n$-tuple of positive integers. Assume that $Z~: \bm{t}\in\Omega\mapsto Z(\bm{t})\subset\R^{n\times n}$ is a bounded  semialgebraic family defined over $\Omega\subset\R^m$. Then, for any $\left(\bm{t}, \bm{H}\right)\in \Omega\times \R_{\ge 0}^n$, it holds that
\begin{equation*}
\left|\frac{\#\left(\left(\A_{\bm{d}}(\bm{H})\right)^n \cap Z(\bm{t})\right)}{\bm{H}^{n\cdot (\bm{d}+\bm{1})}} - \int_{Z(\bm{t})}\prod_{k=1}^{n} \frac{\bm{\rho}_{\bm{d}}\left(\bm{x}^{(k)}\right)\cdot d\bm{x}^{(k)}}{2^n\zeta(\bm{d}+\bm{1})}\right|\;\le\; M_1(Z, \bm{d}, n)\cdot  \max_{1\le i \le n}\frac{(\log H)^{l(d_i)}}{H_i},
%&\qquad \le \; 2^n\cdot\left(\prod_{i=1}^nd_i^2\right)\cdot \left(\max_{1\le i \le n } C_2(d_i)+2^D\max_{1\le k\le d_1\dots d_n}C_2(\bm{d}, k; Z)\right)\cdot \max_{1\le i \le n}\frac{(\log H)^{l(d_i)}}{H_i}\cdotp
\end{equation*}
where the symbol $l(d_i)$ is  defined in~\eqref{defepsikronecker} and where 
\begin{align}
M_1(Z, \bm{d}, n)\;=\; &\left(d_1\dots d_n\right)^n\cdot\left(2^{n^2}\max_{1\le i\le n}E_1(d_i)\;\right.\nonumber\\ 
&\qquad \qquad \quad\quad  \left. + \;2^{n(2D+1)+1}\cdot p(Z,\bm{d})\cdot \left(2\cdot p(Z, \bm{d})-1\right)^{n+s(Z,\bm{d})-1}\right).\label{defM1Zdn}
\end{align} 
In this relation, 
\begin{itemize}
\item given an integer $d\ge 2$, one sets
\begin{equation}\label{vale1d}
E_1(d)\; =\; 6\cdot d^2 2^{d(d-2)}(d+1)^{d/2};
%\begin{cases}
%6 & \textrm{if } d=2;\\
%d^2 2^{d(d-2)}(d+1)^{d/2} & \textrm{if } d\ge 3.
%\end{cases}
\end{equation}
\item the integers $s(Z, \bm{d})$ and $p(Z, \bm{d})$ are defined as
\begin{equation}\label{defsZdpZd}
s(Z, \bm{d})= 3S^{2^{D+n}}\cdot P^{\left(16^{D+n}-1\right)\cdot \left\lceil \log P/\log 2\right\rceil} \;\;\; \textrm{and}\;\;\; p(Z, \bm{d})=4^{\left(4^{D+n}-1\right)/3}\cdot P^{4^{D+n}},
 \end{equation}
where  
\begin{equation}\label{deffD}
D=\sum_{i=1}^{n}\left(d_i+1\right)
\end{equation} 
and where 
\begin{equation*}\label{defdelta}
S\;=\; s(Z)+n\cdot (D+(d_1\dots d_n)^n+1) \qquad \textrm{and}\qquad P\;=\; \max\left\{p(Z), \max_{1\le j\le n} d_j\right\}.
\end{equation*}
\end{itemize}
\end{thm}

\noindent  Theorem~\ref{introeffesubs} stated in the introduction is first deduced from the above theorem. This deduction relies on a lattice point counting estimate valid for semialgebraic families which is  essentially due to Davenport~\cite{dav1, dav2}. Versions of it extended to the larger class of $o$--minimal structures have been obtained by Baroero \& Widmer~\cite{BW2014}. The following lemma, however, presents the advantage of making all parameters completely effective.

\begin{lem}[Davenport's Lemma for Semialgebraic Families]\label{BaWid14}
Let $m,n\ge 0$ be integers and let $Y^*\subset \R^{m+n}$ be the graph of a semialgebraic family with bounded fibers $Y(\bm{t})\subset\R^n$ for all $\bm{t}\in\R^m$. Then, there exists a constant $C(Y)>0$ depending only on the family $Y$ such that for any $\bm{t}$,
\begin{equation}\label{barwidesti}
\left|\#\left(Y(\bm{t})\cap \Z^n\right) - \V_n(Y(\bm{t}))\right|\;\le\; \sum_{j=0}^{n-1}C(Y)^{n-j}\cdot V_j(Y(\bm{t})),
\end{equation}
where $V_0(\,\cdot\,)=1$ and where for $1\le j\le n-1$, the quantity $V_j(\,\cdot\,)$ denotes the sum of the $j$--dimensional Lebesgue measures of all orthogonal projections of its argument on every $j$-dimensional coordinate subspace. Furthermore, an admissible value for the constant $C(Y)$ is 
\begin{equation}\label{valeucst}
C(Y)\;=\; p(Y)\cdot(2p(Y)-1)^{n+s(Y)-1},
\end{equation}
where $p(Y)$ and $s(Y)$ are defined as in~\eqref{defpsZ}.
\end{lem} 

%\noindent In the above, it should be noted that the quantity $s$ can also be upper bounded uniformly over all semialgebraic subsets in $\R^{m+n}$ as a function of   $m+n$ alone --- this is the content of the Bröcker--Scheiderer Theorem (see~\cite[Theorems~4.2.7 \& 4.2.9]{scheiderer}).

\begin{proof} Let $\left(\bm{e}_1, \dots, \bm{e}_n\right)$ denote the canonical basis of $\R^n$. Davenport~\cite{dav1, dav2} shows that the estimate~\eqref{barwidesti} holds when $C(Y)$ is replaced with a constant depending on the pa\-ra\-me\-ter $\bm{t}$, say $ C(Y, \bm{t})$. This constant is defined as  the maximum over all nonempty subsets $I\subset \left\llbracket 1, n\right\rrbracket$, all indices $i_0\in I$ and all lines $l(I, i_0)$ parallel to $\bm{e}_{i_0}$ lying in the span of the family $\mathcal{F}(I)=\left\{\bm{e}_i : i\in I\right\}$, of the number of connected components of $\pi_{I}(Z(\bm{t}))\cap l(I, i_0)$. Here, $\pi_I$ denotes the orthogonal projection from $\R^n$ to $ \mathcal{F}(I)$. \\

\noindent It is furthermore known, see~\cite[Proposition~4.13]{coste}, that $C(Y, \bm{t})$ defined this way is bounded above by $p'(2p'-1)^{k'+s'-1}$ when the fibre $Y(\bm{t})$ is defined by a system of $s'$ polynomial equations and inequalities in $k'$ variables of degrees at most $d'$. Since such a system can be obtained from a polynomial system defining the graph set $Y^*\subset \R^{m+n}$ by freezing the vector variable $\bm{t}\in\R^m$, it is clear that, in the notations of the statement, one can assume that $s'\le s$, $k'\le n$ and $p'\le p$. This completes the proof.
\end{proof}

\noindent Fix an admissible parametrisation $\bm{\varphi}$ as in~\eqref{mapalgparam} (recall that by assumption, its ima\-ge is  assumed to be bounded). Theorem~\ref{thm1.1} is applied to the family of sets parametrised by the map 
\begin{equation}\label{varphi}
W_{\bm{\varphi}}~: Im(\bm{\varphi})\times\R_{\ge 0}\;\rightarrow\; \mathcal{P}(Im(\bm{\varphi})^n)
\end{equation}
which to $\bm{t}=\left(\bm{y}, \eta\right)\in \Omega' = Im(\bm{\varphi})\times\R_{\ge 0}$ associates the set $$\left\{\left(\bm{x}_1, \dots, \bm{x}_n\right)\in Im(\bm{\varphi})^n\; :\; \prod_{i=1}^{n}  \delta\left(\bm{\varphi}^{-1}(\bm{x}_i), \;\bm{\varphi}^{-1}(\bm{y})^{\perp}\right)\le \xi\left(\bm{\varphi}^{-1}(\bm{x}_1), \dots, \bm{\varphi}^{-1}(\bm{x}_n)\right)\cdot\eta\right\}.$$ Here, $\delta$ is the projective distance  defined in~\eqref{defprojdistan} and~\eqref{defprojdistanbis} and $\xi$ is the ortho\-go\-na\-lity defect defined in~\eqref{deforthodefct}. It then follows from a routine application of the Tarski--Seidenberg Theorem~\cite[\S 2.1.2]{coste} that the map $W_{\bm{\varphi}}$ determines a semialgebraic family parametrised by $\Omega'$ whose fibers are furthermore clearly bounded. The volumes of these fibers are now estimated in order to apply Theorem~\ref{thm1.1}. %From now on, it will be more convenient to denote the $k$--dimensional Lebesgue measure  by $\V_k$.

\begin{lem}\label{volestim}
Consider an admissible parametrisation $\bm{\varphi}$ as in~\eqref{mapalgparam}. Then, there exists a constant $M_2(\bm{\varphi}, n)>0$ such that for all $\bm{y}\in Im(\bm{\varphi})$ and all $\eta\in (0, 1/2)$, it holds that 
\begin{equation}\label{volnn-1}
\V_{n(n-1)}\left(W_{\bm{\varphi}}(\bm{y}, \eta)\right)\;\le\; M_2(\bm{\varphi}, n)\cdot \eta\cdot \left|\log \eta\right|^{n-1}.
\end{equation} 
Furthermore, in the case that $\bm{\varphi}$ is the stereographic projection $\bm{\sigma}$ defined in~\eqref{stereo}% or the affine--type one $\bm{\tau}$ defined in~\eqref{deftauparam}
, one can set
\begin{equation}\label{cstspecparma}
M_2(\bm{\sigma}, n)\;=\; 2^{n+1}\cdot v_{n-2},
\end{equation}
where $v_{n-2}$ denotes the surface measure of the unit sphere $\Sph^{n-2}$ (conventionally, $v_0=1$).
\end{lem}

\begin{proof}
It suffices to prove the upper bound~\eqref{volnn-1} for the %is enough to prove the upper bound~\eqref{volnn-1} in the case that the set $W_{\bm{\varphi}}(\bm{y}, \eta)$ is replaced with the 
volume of the superset $$\widehat{W}_{\bm{\varphi}}(\bm{y}, \eta)\;=\; \left\{\left(\bm{x}_1, \dots, \bm{x}_n\right)\in Im(\bm{\varphi})^n\; :\; \prod_{i=1}^{n}  \delta\left(\bm{\varphi}^{-1}(\bm{x}_i), \;\bm{\varphi}^{-1}(\bm{y})^{\perp}\right)\le  \eta\right\}.$$  Fix the integer 
\begin{equation}\label{intA}
A\;=\; \left\lceil -\log_2\eta\right\rceil
\end{equation}
and, given $k\ge 1$, set $$\widehat{V}_{\bm{\varphi}}(\bm{y}, k)\;=\; \left\{ \bm{x}\in Im(\bm{\varphi})\; :\; \delta\left(\bm{\varphi}^{-1}(\bm{x}), \;\bm{\varphi}^{-1}(\bm{y})^{\perp}\right)\in \left(2^{-(k+1)}, 2^{-k}\right]\right\}.$$ Then, clearly, $$\widehat{W}_{\bm{\varphi}}(\bm{y}, \eta)\;\subset\; \bigsqcup_{\underset{k_1+\cdots +k_n\ge A}{k_1, \dots, k_n\ge 0}}\prod_{i=1}^{n} \widehat{V}_{\bm{\varphi}}(\bm{y}, k_i).$$

\noindent Supposing that there exists a constant $M'_2(\bm{\varphi}, n)>0$ such that  
\begin{equation}\label{m'2phin}
\V_{n-1}\left(\widehat{V}_{\bm{\varphi}}(\bm{y}, k\right)\;\le\; M'_2(\bm{\varphi}, n)\cdot 2^{-k}
\end{equation}
yields 
\begin{align*}
\V_{n(n-1)}\left(\widehat{W}_{\bm{\varphi}}(\bm{y}, \eta)\right)\;&\le\; \left(M'_2(\bm{\varphi}, n)\right)^n\cdot \sum_{\underset{k_1+\cdots +k_n\ge A}{k_1, \dots, k_n\ge 0}}2^{-(k_1+\cdots +k_n)}\\
&\le \; \left(M'_2(\bm{\varphi}, n)\right)^n \cdot A^{n-1}\cdot 2^{-A}\\
&\underset{\eqref{intA}}{\le}\; \left(2\cdot M'_2(\bm{\varphi}, n)\right)^n \cdot  \eta\cdot \left|\log \eta\right|^{n-1}, 
\end{align*}
which concludes the  proof of the first part of the statement.  To establish~\eqref{m'2phin}, identify the domain of definition of the map $\bm{\varphi}$ with a subset of the hemisphere through the relation of antipodal equivalence and  set $\bm{u}=\bm{\varphi}^{-1}(\bm{y})\in\Sph^{n-1}$. Consider the manifold $\widehat{V}_{\bm{\varphi}}(\bm{y}, \infty)$ defined as the set of $\bm{x}\in Im(\bm{\varphi})$ such that $\bm{\varphi}^{-1}(\bm{x})\in \bm{u}^{\perp}$ (i.e.~such that $\bm{\varphi}^{-1}(\bm{x}) \bm{\cdot u}=0$). The uniform directional re\-gu\-la\-rity assumption~\eqref{polynalgparambis} implies that  $\widehat{V}_{\bm{\varphi}}(\bm{y}, \infty)$ is regular (in the sense that the gradient vector does not vanish over it). A classical result by Federer~\cite[Theorem~4.12]{fed} (phrased in the language of the \emph{reach} of a manifold) then states the existence of a real $\rho=\rho(\bm{\varphi})>0$ (independent from $\bm{u}$ from the uniformity assumption) such that the projection map from the $\rho$-neighbourhood of $\widehat{V}_{\bm{\varphi}}(\bm{y}, \infty)$ onto it is well--defined. Denote this neighbourhood by $\mathcal{N}_{\bm{\varphi}}(\bm{y}, \rho)$.\\

\noindent Under the assumption of the existence of the real $\rho>0$, it is an easy consequence of the Tubular Neighbourhood Theorem (see~\cite[Theorem~6.24]{lee} and its proof) that there exists another real  $r=r(\bm{\varphi})>0$ such that the map $\left(\bm{x}, \bm{t}\right)\mapsto \bm{x}+\bm{t}$ realises a diffeomorphism between the normal bundle of the manifold $\widehat{V}_{\bm{\varphi}}(\bm{y}, \infty)$ (\textsuperscript{9}\let\thefootnote\relax\footnotetext{\textsuperscript{9} {Recall that this is the set of all pairs $\left(\bm{x}, \bm{t}\right)$ such that $\bm{x}\in \widehat{V}_{\bm{\varphi}}(\bm{y}, \infty)$ and $\bm{t}$ lies in the subspace orthogonal to the tangent space to $\widehat{V}_{\bm{\varphi}}(\bm{y}, \infty)$ at the point $\bm{x}$.}}) and its neighbourhood $\mathcal{N}_{\bm{\varphi}}(\bm{y}, r)$. In particular, there exists a constant  $M''_2(\bm{\varphi})>0$ (which can be effectively computed from the arguments following Equation (6.47)  in~\cite{moi}) such that for any  $\varepsilon\in \left(0, r\right)$, one has $\V_{n-1}\left(\mathcal{N}_{\bm{\varphi}}(\bm{y}, \varepsilon)\right)\le M''_2(\bm{\varphi})\cdot  \varepsilon$. To obtaint~\eqref{m'2phin}, it then remains to observe that the bi--Lipschitz assumption made on the map $\bm{\varphi}$ implies the existence of yet another constant  $M'''_2(\bm{\varphi})>0$ such that for all $k\ge 1$, the set $\widehat{V}_{\bm{\varphi}}(\bm{y}, k)$ is contained in $\mathcal{N}_{\bm{\varphi}}\left(\bm{y},M'''_2(\bm{\varphi})\cdot 2^{-k}\right)$.\\

\noindent Finally, when specialising to the case of the stereographic projection $\bm{\varphi}=\bm{\sigma}$, elementary geometric considerations show that the volume of the set $\widehat{V}_{\bm{\sigma}}(\bm{y}, k)$ is maximised by the volumle of the annulus $\widehat{V}_{\bm{\sigma}}(\bm{e}_0, k-1)$, where,  in the notations of the introduction (see~\S\ref{heightsdegalgvec}), $\bm{e}_0$ denotes the "north pole". Short calculations show that this is  bounded above by $M'_2(\bm{\sigma}, n)\cdot 2^{-k}$, where $M'_2(\bm{\varphi}, n)= 2^{n+1}\cdot v_n$.
\end{proof}

\begin{proof}[Completion of the proof of Theorem~\ref{introeffesubs}] Given $\Lambda$ a line lying in the domain of definition $\mathfrak{S}$ of the map $\bm{\varphi}$ and given $\eta\in (0,1)$,  set  $$\Delta_n(\Lambda, \eta)\;=\; \left\{\left(L_1, \dots, L_n\right)\in\left(\Pj_{n-1}(\R)\right)^n\;:\; \prod_{i=1}^{n}\delta(L_i, \Lambda^\perp)\;\le\; \xi(L_1, \dots, L_n)\cdot \eta\right\}.$$ %where $ \xi(L_1, \dots, L_n)$ is the orthogonality defect defined in~\eqref{deforthodefct}. \\
Given vectors of degrees  $\bm{d}\in\N_{\ge 2}^{n-1}$ and heights $\bm{H}\in\R_{\ge 0}^{n-1}$, it follows from the definitions of the sets $P_{\bm{\varphi}}(\bm{d}, \bm{H})$ and $\A_{\bm{d}}(\bm{H})$ in~\eqref{defboundegheightalflin} and~\eqref{prodalg}, respectively, that $$\#\left(P_{\bm{\varphi}}(\bm{d}, \bm{H})^n\cap \Delta_n(\Lambda, \eta)\right)\;=\; \#\left(\A_{\bm{d}}(\bm{H})^n\cap W_{\bm{\varphi}}\left(\bm{\varphi}(\Lambda), \eta\right)\right).$$ From Theorem~\ref{thm1.1}, there exists a constant 
\begin{equation}\label{hatcdphin}
\widehat{C}_n(\bm{d}, \bm{\varphi})>0,
 \end{equation}
which can be taken as the right--hand side of~\eqref{defM1Zdn} in the case that  the semialgebraic family is the map~\eqref{varphi},  such that 
\begin{align*}
&\left| \frac{\#\left(P_{\bm{\varphi}}(\bm{d}, \bm{H})^n\cap \Delta_n(\Lambda, \eta)\right)}{\bm{H}^{n(\bm{d}+\bm{1})}} - \int_{W_{\bm{\varphi}}\left(\bm{\varphi}(\Lambda), \eta\right)} \left(\prod_{i=1}^{n}\frac{\bm{\rho}_{\bm{d}}\left(\bm{x}^{(i)}\right)\cdot d\bm{x}^{(i)}}{2^n\zeta(\bm{d}+\bm{1})}\right) \right|\\
&\qquad\qquad\qquad\qquad\qquad\qquad\qquad\qquad\qquad\qquad \le\; \widehat{C}_n(\bm{d}, \bm{\varphi})\cdot \max_{1\le i\le n }  \frac{(\log H)^{l(d_i)}}{H_i}\cdotp
\end{align*}
%Furthermore, an admissible value for this constant is
%\begin{align*}
%\widehat{C}_n(\bm{d}, \bm{\varphi})\;=\; ????
%\end{align*}
\noindent In particular, 
\begin{align}
\frac{\#\left(P_{\bm{\varphi}}(\bm{d}, \bm{H})^n\cap \Delta_n(\Lambda, \eta)\right)}{\bm{H}^{n(\bm{d}+\bm{1})}}\;&\le\;\widetilde{C}_n(\bm{d})\cdot \V_{n(n-1)}\left(W_{\bm{\varphi}}\left(\bm{\varphi}(\Lambda), \eta\right)\right)\nonumber \\
&\qquad  +\widehat{C}_n(\bm{d}, \bm{\varphi})\cdot \max_{1\le i\le n }  \frac{(\log H)^{l(d_i)}}{H_i}\label{majestimcrp}
\end{align}
for some constant $\widetilde{C}_n(\bm{d})>0$ which,  from~\eqref{boundkoledpdf}, can be taken as
\begin{equation}\label{tildecndphi}
\widetilde{C}_n(\bm{d})\;=\; \left(\prod_{i=1}^{n}\frac{d_i(d_i+1)}{2^{n+1}\zeta(d_i+1)}\right)^n.
\end{equation}

\noindent In order to apply this estimate to the case of the approximation set $J_\Psi\left(k; Q_1, Q_2\right)$ defined in~\eqref{defjpsikq1q2}, fix an $n$--tuple of lines $\left(L_1, \dots, L_n\right)\in J_\Psi\left(k; Q_1, Q_2\right)$. By assumption, there exists a rational subspace $\Pi\in G_{k,n}(\Q)$ with $H(\Pi)\ge Q_1$ containing $k$ linearly independent rational lines $\Lambda_1, \dots, \Lambda_k$ such that $H(\Lambda_1)\le \dots \le H(\Lambda_k)\le Q_2$ satisfying the Generalised Subspace Inequality~\eqref{subspbis}. Consider then the $k$--dimensional primitive lattice $\mathcal{L}=\Pi\cap\Z^n$ and its successive minima $\lambda_1(\mathcal{L}) \le \dots\le \lambda_k(\mathcal{L})$. Then, $\lambda_i(\mathcal{L})\le H(\Lambda_i)\le Q_2$ for all $1\le i\le k$. Since from Minkowski's Second Convex Body Theorem, it holds that $$\frac{2^k}{k!}\cdot H(\Pi)\;:=\; \frac{2^k}{k!}\cdot \det(\mathcal{L})\;\le\; \omega_k\cdot\prod_{i=1}^{k}\lambda_i(\mathcal{L})$$ (recall that $\omega_k$ denotes the volume of the $k$--dimensional Euclidean ball), one has in particular that $$Q_1^{1/k}\;\le\; H(\Pi)^{1/k}\;\le\; \frac{(k!\cdot \omega_k)^{1/k}}{2}\cdot H(\Lambda_k).$$ As a consequence, $$J_\Psi\left(k; Q_1, Q_2\right)\;\subset\; \bigcup_{\underset{2(Q_1/(k!\cdot \omega_k))^{1/k}\le H(\Lambda)\le Q_2}{\Lambda\in G_{1,n}(\Q)}}\Delta_n\left(\Lambda, \frac{\Psi(H(\Lambda))}{H(\Lambda)^n}\right).$$

\noindent The upper bound~\eqref{majestimcrp} then implies that 
\begin{align*}
&\frac{\#\left(P_{\bm{\varphi}}(\bm{d}, \bm{H})^n\cap J_\Psi\left(k; Q_1, Q_2\right)\right)}{\bm{H}^{n(\bm{d}+\bm{1})}}\\
&\le\; \sum_{\underset{2(Q_1/(k!\cdot \omega_k))^{1/k}\le H(\Lambda)\le Q_2}{\Lambda\in G_{1,n}(\Q)}}\left( \widetilde{C}_n(\bm{d})\cdot \V_{n(n-1)}\!\left(W_{\bm{\varphi}}\left(\bm{\varphi}(\Lambda), \frac{\Psi(H(\Lambda))}{H(\Lambda)^n}\right)\right) \right.\\
&\left.  \qquad\qquad \qquad \qquad \qquad \qquad \qquad \qquad\qquad \qquad +\;\widehat{C}_n(\bm{d}, \bm{\varphi})\cdot \max_{1\le i\le n }  \frac{(\log H)^{l(d_i)}}{H_i}\right).
\end{align*}

\noindent Since the set of rational lines $\Lambda$ of a given height $Q$ is, up to a sign factor, in bijection with the set of primitive integer points in the Euclidean sphere of radius $Q$ centered at the origin, its cardinality is bounded above by the number of integer points on this sphere. Denoting by $B_n(Q)$ the closed $n$--dimensional Euclidean ball of radius $Q$ centered at the origin,  it is easily deduced from  Lemma~\ref{BaWid14} (applied with the choice of parameters $p=2$ and $s=1$) that this number can be estimated as $$\#\left(\Z^n\cap B_n(Q)\right) - \#\left(\Z^n\cap B_n(Q-1)\right)\;\le\; C'_n\cdot Q^{n-1}, $$ where 
\begin{equation}\label{defc'n}
C'_n\;=\; 2^n\omega_n+n\cdot\max_{0\le k\le n-1}\left(2\cdot 3^n\right)^{n-k}\cdot\omega_k.
\end{equation}

\noindent The volume estimate established in Lemma~\ref{volestim} then yields that 
\begin{align*}
&\frac{\#\left(P_{\bm{\varphi}}(\bm{d}, \bm{H})^n\cap J_\Psi\left(k; Q_1, Q_2\right)\right)}{\bm{H}^{n(\bm{d}+\bm{1})}}\\
&\le\; \left(\sum_{ 2(Q_1/(k!\cdot \omega_k))^{1/k}\le Q\le Q_2 }C'_n\cdot\widetilde{C}_n(\bm{d})\cdot M_2(\bm{\varphi}, n)\cdot \frac{\Psi(Q)}{Q}\cdot \left|\log\left(\Psi(Q)\cdot Q^{-n}\right)\right|^{n-1} \right)\\
&  \qquad +\;C'_n\cdot \widehat{C}_n(\bm{d}, \bm{\varphi})\cdot Q_2^n\cdot \max_{1\le i\le n }  \frac{(\log H)^{l(d_i)}}{H_i}\cdotp
\end{align*}
\noindent Under the assumption that $\Psi(Q)\ge Q^{-n}$, this concludes the proof of Theorem~\ref{introeffesubs} upon setting (in the notations of the statement) 
\begin{equation}\label{cndvarphi}
C_n(\bm{d}, \bm{\varphi})\;=\; C'_n\cdot\left((2n)^{n-1}\cdot \widetilde{C}_n(\bm{d})\cdot M_2(\bm{\varphi}, n)+ \widehat{C}_n(\bm{d}, \bm{\varphi})\right).
\end{equation}
Here,  $M_2(\bm{\varphi}, n)$ is the multiplicative constant introduced in~\eqref{volnn-1}, $\widehat{C}_n(\bm{d}, \bm{\varphi})$ is defined in~\eqref{hatcdphin}, $\widetilde{C}_n(\bm{d})$ in~\eqref{tildecndphi} and $C'_n$ in~\eqref{defc'n}. \\

\noindent In the case of the stereographic projection $\bm{\sigma}$, Lemma~\ref{volestim} provides an explicit value for the real $M_2(\bm{\sigma}, n)$. Since the definitions of the projective distance $\delta$ in~\eqref{defprojdistan}, of the orthogonality defect in~\eqref{deforthodefct} and of the inverse stereographic  projection in~\eqref{stereocoords} all involve quadratic polynomials, the parameters $p(Z)$ and $s(Z)$ appearing in the definition of $\widehat{C}_n(\bm{d}, \bm{\varphi})$ (see~\eqref{defM1Zdn} and~\eqref{hatcdphin} and recall that the  corresponding graph $Z^*$ is here that of the semialgebraic map~\eqref{varphi} when $\bm{\varphi}=\bm{\sigma}$) can respectively be taken as 
\begin{equation}\label{defpZsZ}
p(Z)=4n+1\qquad \textrm{ and }\qquad s(Z)=n+1.
\end{equation} 
From the explicit value for $M_2(\bm{\sigma}, n)$ stated in~\eqref{cstspecparma}, an admissible value for $C_n(\bm{d}, \bm{\sigma})$ then turns out to be 
\begin{equation}\label{cndsigma}
C_n(\bm{d}, \bm{\sigma})\;=\; C'_n\cdot\left(4(4n)^{n-1}\cdot \widetilde{C}_n(\bm{d})\cdot  \nu_{n-2}+ \widehat{C}_n(\bm{d}, \bm{\sigma})\right),
\end{equation}
where $\nu_{n-2}$ is the surface measure of $\Sph^{n-2}$, $\widetilde{C}_n(\bm{d})$ and $C'_n$ are again defined  in~\eqref{tildecndphi} and  in~\eqref{defc'n} respectively, and where $\widehat{C}_n(\bm{d}, \bm{\sigma})$ is   the parameter $M_1(Z, \bm{d}, n)$ introduced in~\eqref{defM1Zdn} specialised to the values~\eqref{defpZsZ}. This completes the proof.
\end{proof}

\subsection{Proof of the Statement on the Distribution of Algebraic Vectors}

\noindent The goal in this final section is to prove Theorem~\ref{thm1.1}. To this end, keep the notations therein and, given a subset $A\subset\R^{n\times n}$ and an integer $\nu\ge 1$, denote by $\Gamma_\nu\left(\bm{d}, \bm{H}, A\right)$ the set of all real tuples 
\begin{equation*}%\label{deffD}
\left(a_0^{(ij)}, \dots, a_{d_j}^{(ij)}\right)_{1\le i,j\le n}\in\R^{n\times D}\backslash\{\bm{0}\}, %\qquad \textrm{where}\qquad D=\sum_{i=1}^{n}\left(d_i+1\right),
\end{equation*} 
(where the integer $D$ is defined in~\eqref{deffD}) satisfying the property that $$\max_{\underset{1\le i\le n}{0\le k\le d_j}}\left|a_{k}^{(ij)}\right|\;\le\; H_j\qquad \textrm{for all}\qquad 1\le j\le n$$ and that $$\#\left\{\left(x_{ij}\right)_{1\le i,j\le n}\in A\; :\; \forall 1\le i, j\le n, \; \sum_{k=0}^{d_j}a_k^{(ij)}x_{ij}^k=0 \right\}\;=\; \nu.$$

\noindent Clearly, the set  $\Gamma_\nu\left(\bm{d}, \bm{H}, A\right)$ is empty when $\nu>\left(d_1\dots d_n \right)^n$. Let also  $U_\nu\left(\bm{d}, \bm{H}, A\right)$ be the set of all integer points $\left(a_0^{(ij)}, \dots, a_{d_j}^{(ij)}\right)_{1\le i,j\le n}\in \Z^{n\times D}\cap \Gamma_\nu\left(\bm{d}, \bm{H}, A\right)$ such that for all $1\le i,j\le n$, $$\gcd\left(a_0^{(ij)}, \dots, a_{d_j}^{(ij)}\right)\;=\; 1\qquad \textrm{and}\qquad \sum_{k=0}^{d_j}a_k^{(ij)}X_{ij}^k\in\Z\left[X\right]\;\;\textrm{ is irreducible.}$$ Plainly, 
\begin{equation}\label{adhcarUdH}
\#\left(\left(\A_{\bm{d}}(\bm{H})\right)^n\cap A\right)\;=\; \sum_{1\le \nu\le \left(d_1\dots d_n\right)^n}\nu\cdot \# U_\nu\left(\bm{d}, \bm{H}, A\right).
\end{equation}
\noindent The following lemma provides an  estimate for the cardinality of each of the sets $U_\nu\left(\bm{d}, \bm{H}, A\right)$. To state it, denote by $\Z_{prim}^{n+1}$ the set of primitive points in $\Z^{n+1}$.

\begin{lem}\label{lem4.2}
With the above notations,
\begin{align*}
&\left|\#U_\nu\left(\bm{d}, \bm{H}, A\right) - \#\left( \Gamma_\nu\left(\bm{d}, \bm{H}, A\right)\cap \prod_{i=1}^{n}\Z_{prim}^{n (d_i+1)}\right)\right|\\
&\qquad  \qquad   \qquad  \qquad  \qquad \qquad \qquad \qquad \le \; 2^{n^2}\bm{H}^{n\cdot (\bm{d}+\bm{1})}\cdot \max_{1\le i \le n} \left\{E_1(d_i)\cdot \frac{\left(\log H_i\right)^{l(d_i)}}{H_i}\right\},
\end{align*}
where the quantities $E_1(d_i)$ are defined in~\eqref{vale1d}.
\end{lem}

\begin{proof}
Given integers $d\ge 2$ and $H\ge 3$, it is proved in~\cite[Lemma~2.2]{koleda} that the cardinality of the set $\mathcal{R}_d(H)$ of integral polynomials of height at most $H$  \emph{reducible}  over $\Q$ satisfies the bound $$\# \mathcal{R}_d(H)\;\le\; E_1(d)\cdot H^d\cdot \left(\log H\right)^{l(d)}$$ for some constant $E_1(d)>0$. The proof of this bound also makes it clear that an admissible  value for the constant is   given by~\eqref{vale1d}.  The lemma immediately follows from this estimate.
\end{proof}

\noindent The next step is to estimate %the quantity $ \#\left( \Gamma_\nu\left(\bm{d}, \bm{H}, A\right)\cap \prod_{i=1}^{n}\Z^{d_i+1}\right)$ and to then derive from this an estimate for 
the cardinality of the set $ \Gamma_\nu\left(\bm{d}, \bm{H}, A\right)\cap \prod_{i=1}^{n}\Z_{prim}^{n(d_i+1)}$. To this end, for all that remains, the set $A$ will be taken to be a fibre of the semialgebraic family $Z\; : \Omega\rightarrow \mathcal{P}(\R^{n\times n })$ introduced in the statement of Theorem~\ref{thm1.1}. For a fixed value of $\nu\ge 1$, define also the  map 
\begin{equation}\label{defwdnusemialg}
W_{\bm{d}, \nu}\;:\; \left(\bm{t}, \bm{H}\right)\in \Omega\times \R_{\ge 0}^n\;\mapsto \;  \Gamma_\nu\left(\bm{d}, \bm{H}, Z(\bm{t})\right). 
\end{equation}
By a standard application of the Tarski--Seidenberg Theorem~\cite[\S 2.1.2]{coste}, this map determines a semialgebraic family with parameters in the space $\Omega'=\Omega\times \R_{\ge 0}^n$.   An effective version of this claim is needed~:

\begin{lem}\label{L}
Keep the above notations. Given an integer $1\le \nu\le (d_1\dots d_n)^n$, % and parameters $\left(\bm{t}, \bm{H}\right)\in \Omega'$, the semialgebraic set $ \Gamma_\nu\left(\bm{d}, \bm{H}, Z(\bm{t})\right)$ can be defined by at most 
the graph of the semialgebraic family $W_{\bm{d}, \nu}$ can be defined by a system of at most $s(Z, \bm{d}) \ge 1$ poynomial equalities and inequalities involving polynomials of degrees at most $p(Z, \bm{d})\ge 1$, where the integers $s(Z, \bm{d})$ and  $p(Z, \bm{d})$ are defined in~\eqref{defsZdpZd}. 
\end{lem}

\begin{proof}
Let $\mathcal{S}(\bm{d}, \nu)$ denote the semialgebraic set of tuples $\left(\bm{t}, \bm{H}, \bm{a}^{(1)}, \dots, \bm{a}^{(n)}, \bm{x}_1, \dots, \bm{x}_\nu\right)\in \Omega\times \R_{\ge 0}^n\times \R^D\times \R^{n\nu}$ satisfying the following three conditions~: (1)  $\left\|\bm{a}^{(j)}\right\|_{\infty}\le H_j$ for each $1\le j\le n$; (2) $\sum_{k=0}^{d_i}a_k^{(j)}x^k_{il}=0$  for each $1\le l\le \nu$ and each $1\le j\le n$ and (3) $\bm{x}_1, \dots, \bm{x}_\nu\in Z(\bm{t})$. Clearly, this is a semialgebraic set defined by a system of at most $s=s(Z)+n\nu+nD$ polynomial relations involving polynomials of degrees at most $p=\max\{p(Z), \Delta\}$, where $\Delta=\max_{1\le j\le n} d_j$.\\

\noindent Let $\mathcal{S}'(\bm{d}, \nu)$ denote the projection of the set $\mathcal{S}(\bm{d}, \nu)$ according to the mapping $\left(\bm{t}, \bm{H}, \bm{a}^{(1)}, \dots, \bm{a}^{(n)}, \bm{x}_1, \dots, \bm{x}_\nu\right)\;\mapsto\; \left(\bm{t}, \bm{H}, \bm{a}^{(1)}, \dots, \bm{a}^{(n)}\right)$. From the effective version of the Tarski--Seidenberg Theorem stated in~\cite[Theorem~1]{proy}, it can be defined by a system of at most $s'(\bm{d}, \nu)$ polynomial relations involving polynomials of degrees at most $p'(\bm{d}, \nu)$, where $$s'(\bm{d}, \nu)= s^{2^{D+n}}\cdot p^{\left(16^{D+n}-1\right)\cdot \left\lceil \log \Delta/\log 2\right\rceil} \qquad \textrm{and}\qquad  p'(\bm{d}, \nu)=4^{\left(4^{D+n}-1\right)/3}\cdot p^{4^{D+n}}.$$

\noindent  Notice that $\mathcal{S}'(\bm{d}, \nu)$ is the union over all integers $\nu\le\mu\le (d_1\dots d_n)^n$ of the sets $$\widehat{W}(\bm{d}, \mu)= \left\{\left(\bm{t}, \bm{H}, \bm{a}^{(1)}, \dots, \bm{a}^{(n)}\right) : \left(\bm{t}, \bm{H}, \right)\in \Omega \times \R_{\ge 0}^n, \left(\bm{a}^{(1)}, \dots, \bm{a}^{(n)}\right)\in \Gamma_{\mu}(\bm{\mu}, \bm{H}, Z(\bm{t}))\right\}.$$ As a consequence, it is easily seen that the graph of the semialgebraic map $W_{\bm{d}, \nu}$, which is the difference set $$\left(\bigcup_{\nu\le\mu\le (d_1\dots d_n)^n}\widehat{W}(\bm{d}, \mu)\right)\backslash\left(\bigcup_{\nu+1\le\mu\le (d_1\dots d_n)^n}\widehat{W}(\bm{d}, \mu)\right)\;=\; \mathcal{S}'(\bm{d}, \nu)\backslash \mathcal{S}'(\bm{d}, \nu+1),$$ can be defined by a system of at most $s'(\bm{d}, \nu)+2s'(\bm{d}, \nu+1)\le 3s'(\bm{d}, (d_1\dots d_n)^n+1)$ polynomial relations involving polynomials of degrees at most $p'(\bm{d}, \nu)$, whence the lemma.
\end{proof}

\begin{coro}\label{lem4.6}
Given $\bm{d}\in\N_{\ge 2}^n$ and $\nu\ge 1$, one has that for all $\left(\bm{t}, \bm{H}\right)\in \Omega\times \R_{>0}^n$, 
\begin{align*}
&\left|\frac{\#\left(W_{\bm{d}, \nu}(\bm{t}, \bm{H})\cap \prod_{i=1}^{n}\Z_{prim}^{n(d_i+1)}\right)}{\bm{H}^{n\cdot(\bm{d}+\bm{1})}} - \frac{\V_{n\times D}\left(W_{\bm{d}, \nu}(\bm{t}, \bm{1})\right)}{\zeta\left(n\cdot(\bm{d}+\bm{1})\right)}\right|\; \le\; \frac{2^{n(2D+1)+1}\cdot C(\bm{d}, Z)}{\min_{1\le i\le n}H_i},%\\
%&\qquad \qquad \qquad \qquad\qquad \qquad\qquad \qquad\qquad \qquad \le\; \frac{2^{nD+1}\cdot\zeta\!\left(n\cdot\left(\bm{d}+\bm{1}\right)-\bm{1}\right)\cdot C(\bm{d}, \nu, Z)}{\min_{1\le i\le n}H_i},
\end{align*}
where $D$ is the integer defined in~\eqref{deffD} and where 
\begin{equation}\label{defcdZ}
C(\bm{d}, Z)\;=\; p(Z,\bm{d})\cdot(2\cdot p(Z,\bm{d})-1)^{n+s(Z,\bm{d})-1}
\end{equation} with the integers $p(Z,\bm{d})$ and $s(Z,\bm{d})$ defined in~\eqref{defsZdpZd}.
\end{coro}

\begin{proof}
From Lemma~\ref{BaWid14} and with the notations therein, 
\begin{equation*}
\left|\#\!\left(W_{\bm{d}, \nu}(\bm{t}, \bm{H})\cap \prod_{i=1}^{n}\Z^{n(d_i+1)}\right)  - \V_{n\times D}\!\left(W_{\bm{d}, \nu}(\bm{t}, \bm{H})\right) \right|\le  C(\bm{d}, \nu, Z)\cdot \sum_{j=0}^{nD-1}\! V_j\!\left(W_{\bm{d}, \nu}(\bm{t}, \bm{H})\right).
\end{equation*}
Here, by homogeneity of the set $W_{\bm{d}, \nu}(\bm{t}, \bm{H})$ in each  component $H_i$, $$\V_{n\times D}\left(W_{\bm{d}, \nu}(\bm{t}, \bm{H})\right)\;=\; \bm{H}^{n\cdot (\bm{d}+\bm{1})}\cdot \V_{n\times D}\left(W_{\bm{d}, \nu}(\bm{t}, \bm{1})\right).$$ Moreover, since $W_{\bm{d}, \nu}(\bm{t}, \bm{1})\subset \left[-1, 1\right]^{n\times D},$
it also holds that 
\begin{align*}
V_j\left(W_{\bm{d}, \nu}(\bm{t}, \bm{H})\right)\;&\le\; V_j\left(\bm{H}^{n\cdot (\bm{d}+\bm{1})}\cdot W_{\bm{d}, \nu}(\bm{t}, \bm{1})\right)\\
&\le\; V_j\left(\bm{H}^{n\cdot (\bm{d}+\bm{1})}\cdot \left[-1, 1 \right]^{n\times D}\right)\;\le\; {nD \choose j}\cdot\frac{2^j\cdot \bm{H}^{\bm{d}+\bm{1}}}{\min_{1\le i\le n} H_i^{nD-j}}\cdotp
\end{align*}
As a consequence, 
\begin{equation}\label{nonprimcount}
\left|\frac{\#\!\left(W_{\bm{d}, \nu}(\bm{t}, \bm{H})\cap \prod_{i=1}^{n}\Z^{n(d_i+1)}\right)}{\bm{H}^{n\cdot (\bm{d}+\bm{1})}}  - \V_{n\times D}\!\left(W_{\bm{d}, \nu}(\bm{t}, \bm{1})\right) \right|\le  \frac{2^{2nD}\cdot C(\bm{d}, \nu, Z)}{\min_{1\le i\le n} H_i}\cdotp
\end{equation}
Returning to the count of primitive points and denoting by $\mu$ the Möbius function, it follows from the inclusion-exclusion principle that 
\begin{align*}
&\frac{\#\left(W_{\bm{d}, \nu}(\bm{t}, \bm{H})\cap \prod_{i=1}^{n}\Z_{prim}^{n(d_i+1)}\right)}{\bm{H}^{n\cdot(\bm{d}+\bm{1})}} \\
&=\; \sum_{\bm{m}=(m_1, \dots, m_n)\in\N_{\ge 1}^n} \left(\prod_{i=1}^{n}\mu(m_i)\right)\cdot \frac{\#\!\left(W_{\bm{d}, \nu}(\bm{t}, \bm{H}/\bm{m})\cap \prod_{i=1}^{n}\Z^{n(d_i+1)}\right)}{\bm{H}^{n\cdot(\bm{d}+\bm{1})}},
\end{align*}
where $\bm{H}/\bm{m}$ denotes the vector with components $H_i/m_i$. From~\eqref{nonprimcount}, one then infers the existence of  reals $\theta, \eta_1, \dots, \eta_n\in [-1, 1]$ such that  
\begin{align*}
&\frac{\#\left(W_{\bm{d}, \nu}(\bm{t}, \bm{H})\cap \prod_{i=1}^{n}\Z_{prim}^{n(d_i+1)}\right)}{\bm{H}^{n\cdot(\bm{d}+\bm{1})}} \\
& \qquad =\; \V_{n\times D}\!\left(W_{\bm{d}, \nu}(\bm{t}, \bm{1})\right)\cdot \sum_{1\le m_1\le H_1}\cdots\sum_{1\le m_n\le H_n}\prod_{i=1}^{n}\frac{\mu(m_i)}{m_i^{n\cdot (d_i+1)}}\\
&\qquad \qquad  + \theta\cdot  2^{2nD}\cdot C(\bm{d}, \nu, Z)\cdot \sum_{1\le m_1\le H_1}\cdots\sum_{1\le m_n\le H_n}\frac{1}{\min_{1\le i\le n}\frac{H_i}{m_i}}\cdot \prod_{i=1}^{n}\frac{1}{m_i^{n\cdot (d_i+1)}}\cdotp
\end{align*}
From the inclusion $W_{\bm{d}, \nu}(\bm{t}, \bm{1})\subset \left[-1, 1\right]^{n\times D}$ and from standard identities on the Riemann zeta function, the main term can be expressed with the help of a parameter $\eta\in [-1, 1]$ as $$\V_{n\times D}\!\left(W_{\bm{d}, \nu}(\bm{t}, \bm{1})\right)\cdot \sum_{\underset{1\le i\le n}{1\le m_i\le H_i}}\prod_{i=1}^{n}\frac{\mu(m_i)}{m_i^{n\cdot (d_i+1)}}\;=\; \frac{\V_{n\times D}\!\left(W_{\bm{d}, \nu}(\bm{t}, \bm{1})\right)}{\zeta{\left(n\cdot\left(\bm{d}+\bm{1}\right)\right)}} + \eta \cdot \frac{2^{2nD}}{\bm{H}^{n\cdot (\bm{d}+\bm{1})}}\cdotp$$
As for the error term, it can be estimated from the inequalities
\begin{align*}
& \sum_{1\le m_1\le H_1}\cdots\sum_{1\le m_n\le H_n}\frac{1}{\min_{1\le i\le n}\frac{H_i}{m_i}}\cdot \prod_{i=1}^{n}\frac{1}{m_i^{n\cdot (d_i+1)}}\\
 &\qquad \le\; \frac{1}{\min_{1\le i \le n}H_i}\cdot  \sum_{1\le m_1\le H_1}\cdots\sum_{1\le m_n\le H_n}\max_{1\le i\le n} m_i \cdot \prod_{i=1}^{n}\frac{1}{m_i^{n\cdot (d_i+1)}}\\
 &\qquad \le\;\frac{\zeta\!\left(n\cdot\left(\bm{d}+\bm{1}\right)-\bm{1}\right)}{\min_{1\le i\le n}H_i}\;\le\; \frac{2^n}{\min_{1\le i\le n}H_i}\cdotp 
\end{align*}
The statement follows upon noticing that, from Lemma~\ref{L}, the integer defined in~\eqref{defcdZ} consitutes a uniform upper bound for the quantities $C(\bm{d}, \nu, Z)$ when $1\le \nu\le (d_1\dots d_n)^n$.
\end{proof}

\begin{proof}[Completion of the proof of Theorem~\ref{thm1.1}]
Combined with the conclusions of Lemma~\ref{lem4.2} and Corollary~\ref{lem4.6}, the definition of the map $W_{\bm{d}, \nu}$ in~\eqref{defwdnusemialg}  and identity~\eqref{adhcarUdH} yield that for all $\left(\bm{t}, \bm{H}\right)\in \Omega\times \R_{\ge 0}^n$, 
\begin{align*}
&\left|\frac{\#\left(\left(\A_{\bm{d}}(\bm{H})\right)^n \cap Z(\bm{t})\right)}{\bm{H}^{n\cdot (\bm{d}+\bm{1})}} -f_{\bm{d}}(\bm{t})\right| \;\le\; M_1(Z, \bm{d}, n)\cdot \max_{1\le i \le n}\frac{(\log H)^{l(d_i)}}{H_i}.
\end{align*}
Here, from Lemma~\ref{lem4.2} and~Corollary\ref{lem4.6},
\begin{equation*}\label{M1Zdn}
M_1(Z, \bm{d}, n)\; =\; \left(d_1\dots d_n\right)^n\cdot\left(2^{n^2}\cdot \max_{1\le i\le n} E_1(d_i) +2^{n(2D+1)+1}\cdot C(\bm{d}, Z)\right)
\end{equation*}
%with the quantities $E_1(d_i) $ and $C(\bm{d}, Z)$ defined in Lemma~\ref{lem4.2} and~Corollary\ref{lem4.6}, respectively, and where
and  
\begin{equation*}
f_{\bm{d}}(\bm{t})\;=\; \sum_{1\le \nu\le (d_1\dots d_n)^n} \nu\cdot \frac{\V_{n\times D}\left(W_{\bm{d}, \nu}(\bm{t}, \bm{1})\right)}{\zeta\left(n\cdot(\bm{d}+\bm{1})\right)}\cdotp
\end{equation*}
The above value of $M_1(Z, \bm{d}, n)$ coincides with the one given in Theorem~\ref{thm1.1}. Furthermore; the function $f_{\bm{d}}$ can be explicitly determined. Indeed, Koleda's Theorem~\ref{koledadensitythm} implies that for any nonempty set $U_n\subset\R^{n\times n}$ defined as a product of one--dimensional intervals, $$ \frac{\#\left(\left(\A_{\bm{d}}(\bm{H})\right)^n \cap U_n)\right)}{\bm{H}^{n\cdot (\bm{d}+\bm{1})}}\;=\; \int_{U_n}  \prod_{k=1}^{n} \frac{\bm{\rho}_{\bm{d}}\left(\bm{x}^{(k)}\right)\cdot d\bm{x}^{(k)}}{2^n\zeta(\bm{d}+\bm{1})}+ O_{\bm{d}, n}\left(\max_{1\le i \le n}\frac{(\log H)^{l(d_i)}}{H_i}\right),$$ where the notation for the error term means that the implicit constant depends only on the degree vector  $\bm{d}$ and on the integer $n$. Since each fiber of $Z$ is a bounded semialgebraic set, it is Jordan measurable. It can thus be approximated arbitrarily well from above and below by finite disjoint unions of product of intervals. Standard arguments from measure theory then yield that any $\bm{t}\in\Omega$, 
\begin{equation*}
f_{\bm{d}}(\bm{t})\;=\; \lim_{\min \bm{H}\rightarrow \infty} \frac{\#\left(\left(\A_{\bm{d}}(\bm{H})\right)^n \cap Z(\bm{t}))\right)}{\bm{H}^{n\cdot (\bm{d}+\bm{1})}}\;=\;  \int_{Z(\bm{t})}  \prod_{k=1}^{n} \frac{\bm{\rho}_{\bm{d}}\left(\bm{x}^{(k)}\right)\cdot d\bm{x}^{(k)}}{2^n\zeta(\bm{d}+\bm{1})},
\end{equation*}
which concludes the proof.
%To conclude the proof, it remains to notice that in the case of the stereographic projection $\bm{\varphi}=\bm{\sigma}$, the multiplying constant in the right-hand side of~\eqref{constantrhs} reduces to the one stated in Theorem~\ref{thm1.1} from Lemma???.
\end{proof}

%\noindent It is known, see~\cite[Theorem~1.1]{koleda}, that there exist constants $C_n( \bm{d},  \bm{\varphi})>0$ and $K_n( \bm{d})>0$ which are explicitly given in Lemma/Proposition?? below such that 
%\begin{equation}\label{crdAdH}
%\left| \#\left( A({\bm{d}}, \bm{H})\cap  \im(\bm{\varphi})\right) - C_n( \bm{d} , \bm{\varphi})\cdot \prod_{j=1}^{n-1}H_j^{n(d_j+1)}\right|\; \le \; K_n( \bm{d})\cdot  \prod_{j=1}^{n-1}\left(H_j^{ d_j}\cdot \left(\log H_j\right)^{ l(d_j)}\right)^n.
%\end{equation} 

\bibliographystyle{unsrt}

%\addcontentsline{toc}{chapter}{\protect\numberline{}References}

\end{document}